\documentclass[a4paper,10pt]{amsart}
\usepackage{amsmath,amssymb,amsxtra,amsfonts,microtype}
\usepackage{amsthm,url,comment,enumerate,comment}

\usepackage{tikz} 
\usetikzlibrary{arrows,backgrounds,decorations,automata}

\theoremstyle{plain}
\newtheorem{theorem}{Theorem}
\newtheorem{lemma}{Lemma}
\newtheorem{corollary}{Corollary}
\newtheorem{proposition}{Proposition}

\theoremstyle{definition}
\newtheorem{remark}{Remark}
\newtheorem{definition}{Definition}
\newtheorem{example}{Example}

\newcommand{\N}{\mathbb{N}}
\newcommand{\R}{{\mathbb R}}
\newcommand{\Z}{{\mathbb Z}}
\newcommand{\Q}{{\mathbb Q}}
\newcommand{\mc}{\mathcal}
\newcommand{\ID}{\mc {D_I}}

\newcommand{\C}{\mathcal}

\newcommand{\gs}{\sigma}

\newcommand{\me}{\rm e}
\newcommand{\mi}{\rm i}
\newcommand{\TT}{\mathbb{T}}
\newcommand{\FT}{\operatorname{FT}}

\makeatletter
\DeclareRobustCommand\bigop[2][1]{%
  \mathop{\vphantom{\sum}\mathpalette\bigop@{{#1}{#2}}}\slimits@
}
\newcommand{\bigop@}[2]{\bigop@@#1#2}
\newcommand{\bigop@@}[3]{%
  \vcenter{%
    \sbox\z@{$#1\sum$}%
    \hbox{\resizebox{\ifx#1\displaystyle#2\fi\dimexpr\ht\z@+\dp\z@}{!}{$\m@th#3$}}%
  }%
}
\makeatother

\newcommand{\bigast}{\DOTSB\bigop{\ast}}

\newcommand{\exend}{\hfill $\Diamond$}

\makeatletter
  \def\vhrulefill#1{\leavevmode\leaders\hrule\@height#1\hfill \kern\z@}
\makeatother

\newcommand{\Span}{\operatorname{span}}

\newcounter{nootje}
\newcommand\noot[1]
 {\marginpar{\small\begin{minipage}{30mm}\begin{flushleft}\thenootje : #1\end{flushleft}\end{minipage}}\addtocounter{nootje}{1}}
\begin{document}

\title[Regular sequences:~parametrisation and spectral characterisation]{Spectral theory of regular sequences:\\ parametrisation and spectral characterisation}


\author{Michael Coons}
\address{Department of Mathematics and Statistics, California State University, \newline
\hspace*{\parindent}400 West First Street, Chico, CA 95926, USA}
\email{mjcoons@csuchico.edu}

\address{Fakult\"at f\"ur Mathematik, Universit\"at Bielefeld, \newline
\hspace*{\parindent}Postfach 100131, 33501 Bielefeld, Germany}
\email{cmanibo@math.uni-bielefeld.de}

\author{James Evans}
\address{School of Information and Physical Sciences, 
   University of Newcastle, \newline
\hspace*{\parindent}University Drive, Callaghan NSW 2308, Australia}
\email{james.evans10@uon.edu.au}

\author{Philipp Gohlke}
\address{Lund University, Centre for Mathematical Sciences, \newline
\hspace*{\parindent}Box 118, 221 00 Lund, Sweden}
\email{philipp\_nicolai.gohlke@math.lth.se}

\author{Neil Ma\~nibo}

\date{\today}

\begin{abstract} We extend the existence of ghost measures beyond nonnegative primitive regular sequences to a large class of nonnegative real-valued regular sequences. In the general case, where the ghost measure is not unique, we show that they can be parametrised by a compact abelian group. For a subclass of these measures, by replacing primitivity with a commutativity condition, we show that these measures have an infinite convolution structure similar to Bernoulli convolutions. Using this structure, we show that these ghost measures have pure spectral type. Further, we provide results towards a classification of the spectral type based on inequalities involving the spectral radius, joint spectral radius, and Lyapunov exponent of the underlying set of matrices. In the case that the underlying measure is pure point, we show that the support of the measure must be a subset of the rational numbers, a result that resolves a new case of the finiteness conjecture. 
\end{abstract}

\maketitle

\section{Introduction}

Let $k\geqslant 2$ be an integer. A real-valued sequence $f$ is called {\em $k$-regular} provided its $k$-kernel, ${\rm ker}_k ({f}):=\left\{(f(k^\ell n+r))_{n\geqslant 0}: \ell\geqslant 0, 0\leqslant r<k^\ell\right\},$ generates a finite dimensional $\mathbb{R}$-vector space $V_k(f)$. The sequence $f$ is {\em $k$-automatic} if and only if ${\rm ker}_k ({f})$ is finite. The dynamical properties of automatic sequences, through their related substitution systems, are well known; see the monographs of Queff\'elec \cite{Qbook} and Baake and Grimm \cite{BGbook}. Towards generalising the dynamical properties of automatic sequences, three of us---Coons, Evans and Ma\~nibo \cite{CEM1}---constructed measures from real-valued primitive regular sequences based on their linear representations, which extended work of Baake and Coons \cite{BC2018}. 

Given a $k$-regular sequence $f$, let $\{f_1,f_2,\ldots,f_d\}\subseteq {\rm ker}_k ({f})$ be a basis for ${V}_k(f)$. 
We call $d = \dim V_k(f)$ the \emph{degree} of $f$.
Set ${\bf f}(n)=(f_1(n),f_2(n),\ldots,f_d(n))^T$ and for each $j\in\{0,\ldots,k-1\}$, let ${\bf A}_j$ be the $d\times d$ real matrix such that, for all $n\geqslant 0$, ${\bf f}(kn+j)^T={\bf f}(n)^T{\bf A}_j.$ 
We refer the reader to the seminal paper of Allouche and Shallit \cite{AS1992} and Nishioka's monograph \cite[Ch.~15]{N1996} for details on existence and the finer definitions. Setting ${\bf u} := {\bf f}(0)$ we have for each $i\in\{1,\ldots,d\}$ and $n>0$ that $$f_i(n)={\bf u}^T{\bf A}_{(n)_k}{\bf e}_i ={\bf u}^T{\bf A}_{i_s}\cdots{\bf A}_{i_{1}}{\bf A}_{i_0}{\bf e}_i,$$ where ${\bf e}_i$ is the $i$th elementary column vector, $(n)_k=i_s\cdots i_1i_0$ is the base-$k$ expansion of $n$ and ${\bf A}_{(n)_k}:={\bf A}_{i_s}\cdots {\bf A}_{i_1}{\bf A}_{i_0}.$ Set $\mathcal{A}:=\{{\bf A}_0,\ldots,{\bf A}_{k-1}\}$ and denote the {\em sum matrix} by ${\bf A}:=\sum_{j=0}^{k-1}{\bf A}_j$. Since $f\in V_k(f)$, there is a ${\bf v}\in\mathbb{R}^{d\times 1}$ such that $f(n)={\bf u}^T {\bf A}_{(n)_k}{\bf v}$ for all $n\geqslant 0$. We call such a tuple $({\bf u},\mathcal{A},{\bf v})$ a {\em  canonical linear representation} of $f$. 

With some abuse of notation, we sometimes write $f=({\bf u},\mathcal{A},{\bf v})$. 
 More generally, for every $d \in \N$, every tuple $({\bf u},\mathcal{A},{\bf v})$ with ${\bf u}, {\bf v} \in \R^{d \times 1}$ and ${\bf A}_i \in \R^{d \times d}$ gives rise to a $k$-regular sequence via the relation $f(n) = {\bf u}^T {\bf A}_{(n)_k} \bf v$, see \cite[Lemma~4.1]{AS1992}. In that case, we call $({\bf u},\mathcal{A},{\bf v})$ a \emph{linear representation} of $f$ and refer to $d$ as the \emph{degree of the representation}.
\begin{remark}
The degree of a linear representation $({\bf u},\mathcal{A},{\bf v})$ of a $k$-regular sequence $f$ can in general be larger (but never smaller) than the degree of $f$. If the two degrees coincide we say that the linear representation is \emph{minimal}. A canonical linear representation is always minimal by construction. It also has the particular properties that ${\bf u}^T = {\bf u}^T {\bf A}_0$ and that $f$ being nonnegative implies that ${\bf u}^T {\bf A}_{(n)_k}$ is a nonnegative vector for all $n \in \N_0$. This can fail for more general linear representations. However, as we will see in Lemma~\ref{LEM:minimal-basis-chage}, every minimal linear representation is related to a canonical one via a change of basis.
\end{remark}

In the following, we assume that $({\bf u},\mathcal{A},{\bf v})$ is a canonical linear representation of a nonnegative $k$-regular sequence $f$. Now, set \begin{equation*}\label{eq:Sigmai}\Sigma_f(N):=\sum_{m=k^N}^{k^{N+1}-1}f(m)\quad\mbox{and}\quad
   \mu^{}_{N} \, := \, \frac{1}{\Sigma_f(N)} \sum_{m=0}^{k^{N+1}-k^N - 1}
   f(k^N + m) \, \delta^{}_{m / k^N(k-1)},
\end{equation*}
where $\delta_x$ denotes the unit Dirac measure at $x$. We can view
$(\mu^{}_{N})^{}_{N\in\mathbb{N}_0}$ as a sequence of probability measures on
the $1$-torus, the latter written as $\mathbb{T}=[0,1)$ with addition modulo
$1$. Here, we have simply re-interpreted the (normalised) values of 
the sequence $(f(n))_{n\geqslant 0}$ between $k^N$ and $k^{N+1}-1$ as the weights of a pure point
probability measure on $\mathbb{T}$ supported on  $\big\{ {m}/{(k^N(k-1))} : 0 \leqslant m <
k^N(k-1) \big\}$. Note that $\mu_N$ is only well-defined if $\Sigma_f(N)$ is nonzero. We impose a slightly stronger assumption that precludes this degenerate situation under an appropriate normalisation. We refer to such sequences as {\em nondegenerate} and take this as a standing assumption for the remainder of this section; see Definition~\ref{def:nondegenerate} for details.

\begin{definition}
We call every accumulation point of $(\mu_N)_{N \in \N_0}$ (in the weak topology) a \emph{ghost measure} of $f$. If there is a unique ghost measure, we call it \emph{the} ghost measure of $f$ and denote it by $\mu_f$. 
\end{definition}

The behaviour of the sequence $(\mu_N)_{N \in \N_0}$ depends on the relation between the spectral radius $\rho(\bf A)$, the \emph{joint spectral} radius of the set of matrices $\mathcal{A}$ defined by
\[
\rho^* =\rho^*(\mathcal{A}):=\lim_{n\to\infty}\max_{0\leqslant i_1,i_2,\ldots,i_{n}\leqslant k-1}\big\| {\bf A}_{i_1}{\bf A}_{i_2}\cdots{\bf A}_{i_{n}}\big\|^{1/n},
\]
and the {\em Lyapunov exponent}, which is the quantity $\log_k\bar{\rho}$, with $$\bar{\rho}=\bar{\rho}(\mathcal{A}):=\lim_{n\to\infty}\Bigg(\prod_{0\leqslant i_1,i_2,\ldots,i_{n}\leqslant k-1}\big\| {\bf A}_{i_1}{\bf A}_{i_2}\cdots{\bf A}_{i_{n}}\big\|^{\frac{1}{n}}\Bigg)^{1/k^n},$$
where $\|\cdot\|$ is any (submultiplicative) matrix norm. 
With this notation, Coons, Evans and Ma\~nibo \cite{CEM1} proved that a positive real-valued $k$-regular sequence $f$ such that the spectral radius $\rho({\bf A})$ is the unique simple maximal eigenvalue of ${\bf A}$ and $\rho^*(\mathcal{A})<\rho({\bf A})$ admits a unique ghost measure $\mu_f$. In the general situation, there is not a unique ghost measure. Nonetheless, there is a cohesive structure of ghost measures which comes from the linear representation.

In this paper, after a short discussion on the basic properties of $k$-regular sequences in Section \ref{sec:prop}, in Section \ref{SEC:param}, we extend and generalise this existence result. In particular, even when the ghost measure is not unique, there is a set of ghost measures that can be parametrised by a compact abelian group $G$. Moreover, the orbit structure of the group dictates among which subsequences the measure approximants $\mu_N$ converge. The group can be obtained by restricting to the $\mathcal{A}$-invariant subspace $V^T=\Span\{{\bf u}^T{\bf A}_{(n)_k}\}_{n\in\mathbb{N}}$ and analysing the peripheral spectrum $\gs_p:=\{g_1\rho,g_2\rho,\ldots,g_s\rho\}$ of $\widetilde{\bf A}={\bf P}_V{\bf A}{\bf P}_V,$ where $\rho=\rho(\widetilde{\bf A})$, each $g_i$ is an element of the complex unit circle $S$ and ${\bf P}_V$ denotes the orthogonal projection to $V$. The compact abelian group $$G=\overline{\{g^n:n\in\mathbb{N}\}}$$ is the subgroup of the $s$-dimensional torus $S^s$ generated by the element $g=(g_1,g_2,\ldots,g_s)$, where the group operation is component-wise multiplication. Denote by $\mc M$ the set of Borel probability measures on the torus $\TT$ endowed with the weak topology.

\begin{theorem}
\label{THM:ghost-measure-group}
Let $f$ be a nonnegative and nondegenerate $k$-regular sequence. 
Then, there is a continuous map $$\Psi_f \colon G \to \mc M$$ that sends $h$ to $\mu_h$, such that the ghost measures of $f$ are given by the set $\Psi_f(G)$.
 Moreover, $ \lim_{j \to \infty } \mu_{n_j} = \mu$ if and only if $\mu = \mu_h$ for every accumulation point $h$ of $(g^{n_j})$ in $G$. 
\end{theorem}

\noindent As an immediate corollary, we have the following result.

\begin{corollary}\label{cor:unique}
If $f$ is a nonnegative $k$-regular sequence and $\rho(\widetilde{\bf A})$ is the unique eigenvalue of maximal modulus of $\widetilde{\bf A}$, then, $f$ has a unique ghost measure $\mu_f$, given by $\mu_f = \lim_{N \to \infty} \mu_N.$
\end{corollary}

\noindent In the setting of nonnegative matrices, the group $G$ is finite, and hence all orbits are periodic. Recall that given a nonnegative matrix ${\bf A}$, there is a number $n\in\N$ such that for each eigenvalue $\lambda$ of ${\bf A}$ with $|\lambda| = \rho({\bf A})$, we have $\lambda^n \in \R$. We call the smallest such number the \emph{period} of $\bf A$. The period coincides with the cardinality of $G$, and so, bounds the number of ghost measures.

\begin{theorem}\label{thm:probnonneg}
Let $f$ be a nonnegative real-valued $k$-regular sequence with canonical representation $({\bf u},\mathcal{A}, {\bf v})$ and assume that all $\widetilde{\bf A}_i$ are nonnegative matrices. Let $p$ be the period of $\widetilde{\bf A}$. Then, there are probability measures $\mu_{f,0},\ldots,\mu_{f,p-1}$ such that
\[
\lim_{n \to \infty} \mu_{pn + j} = \mu_{f,j}, 
\]
for all $0\leqslant j \leqslant p-1$ in the sense of weak convergence. 
\end{theorem}

\noindent Section \ref{SEC:param} contains several other results, including a level-set procedure for constructing ghost measures, as well as the development of a reduced representation of a nonnegative regular sequence.

In Section \ref{sec:spectype}, we study the Lebesgue decomposition of ghost measures. In \cite{CEM1}, under the condition of $\rho^*(\mathcal{A})<\rho(\bf{A})$, it was shown that the ghost measure $\mu_f$ is continuous. We extend this result, undergoing a further reduction to a subspace $\widehat{V}$ that captures the limiting behaviour. We write $\widehat{\bf A}$ and $\widehat{\mc A}$ for the corresponding restrictions of $\bf A$ and $\mc A$ to this subspace; for the precise definitions see \eqref{def:vhat}.  

\begin{theorem}
\label{PROP:continuity-condition}
We have $\rho^*(\widehat{{\mathcal{A}}}) < \rho(\widehat{{\bf A}})$ if and only if all ghost measures $\mu_h$ with $h \in G$ are continuous.
\end{theorem}

\noindent In fact, in this more general context, we can say much more than in \cite{CEM1}. Recall, any finite real Borel measure $\mu$ on $\mathbb{T}$ has a {\em generalised Lebesgue decomposition}; that is, $\mu$ is the sum of three mutually singular measures $\mu_{\rm pp}$, $\mu_{\rm sc}$ and $\mu_{\rm ac}$, where, with respect to Lebesgue measure $\lambda$, $\mu_{\rm pp}$ is pure point (the so-called Bragg part), $\mu_{\rm sc}$ is singular continuous and $\mu_{\rm ac}$ is absolutely continuous. Depending on the strictness of inequalities in the fundamental inequality, $$\bar{\rho}\leqslant \frac{\rho}{k}\leqslant \rho^*\leqslant \rho,$$ which holds for cone-preserving matrix semigroups, we can be more specific. 
For example, if $\bar{\rho}<\rho/k$, then all continuous ghost measures are singular. For a more restricted class, we have the following classification.

\begin{theorem}\label{thm:specclass} Suppose that the matrices in $\widehat{\mathcal{A}}$ are nonnegative, do not share a nontrivial invariant subspace, and that $\rho^*:= \rho^*(\widehat{\mathcal{A}})< \rho:=\rho(\widehat{\bf A})$. Then 
\begin{enumerate}
\item[(i)] $\rho/k<\rho^*$ if and only if all ghost measures are singular continuous, and 
\item[(ii)] $\rho/k=\rho^*$ if and only if all ghost measures are absolutely continuous.
\end{enumerate}
\end{theorem}

Under appropriate conditions, it is possible to assign to each ghost measure $\mu$ of a sequence $f$ a different $k$-regular sequence $f'$ with convenient properties such that $\mu$ is the \emph{unique} ghost measure for $f'$. Such a sequence $f'$ has a \emph{reduced representation} $({\bf w}, \mathcal{B}, {\bf x})$, introduced in Section \ref{SEC:param}.
In this context, in Section \ref{sec:structure}, we show that the ghost measure of a reduced representation has an infinite convolution structure (Proposition \ref{prop:mainKmu}), which can be thought of as a higher-dimensional analogue of Bernoulli convolutions. This convolution structure is then used in Section \ref{sec:ergspecpur} to show that the ghost measure is spectrally pure. 

\begin{theorem}\label{thm:specpure}
Suppose $f$ has a reduced representation $({\bf w},\mathcal{B},{\bf x})$ such that the spectral radius $\rho$ is the unique simple maximal eigenvalue of $\widehat{\bf B}$. Then the unique ghost measure $\mu_f$ is spectrally pure, that is, 
$(\mu_f)=(\mu_f)_{\bullet}$, for some $\bullet\in \left\{\textnormal{pp,sc,ac}\right\}$. 
\end{theorem}

\noindent Note that, in general, spectral purity does not hold, even in the simple setting of nonnegative matrices; see Example \ref{EX:pp+ac}.

In Section \ref{sec:rational}, we study more closely the pure point parts $(\mu_f)_{\textnormal{pp}}$ of measures arising from regular sequences $f$ whose associated set $\widetilde{\mathcal{A}}$ contains only nonnegative matrices. Using a graph-theoretic method inspired by the understanding of the evolution of a Markov process, we show that any pure point of $\mu_f$ must be a rational number. In particular, we have the following result.

\begin{theorem}\label{thm:mainrational} Let $f$ be a nonnegative real-valued $k$-regular sequence with canonical linear representation $f=({\bf u},\mathcal{A},{\bf v})$ and assume that all $\widetilde{\bf A}_i$ are nonnegative matrices and that ${\bf v}$ is nonnegative. Then, the pure point part of every ghost measure is supported on $\mathbb{Q}.$
\end{theorem}

\noindent A finite set of matrices $ \mc A = \{{\bf A}_0,\ldots,{\bf A}_{k-1}\}$ is said to satisfy the {\em finiteness property} provided there is a finite product ${\bf A}_{i_0}\cdots {\bf A}_{i_{m-1}}$ of these matrices such that $$\rho({\bf A}_{i_0}\cdots {\bf A}_{i_{m-1}})^{1/m}=\rho^*(\mc A).$$ Arising from the work of Daubechies and Lagarias \cite{DL1991,DL1992}, Lagarias and Wang \cite{LW1995} conjectured that the finiteness property holds for all finite sets of real matrices. This was shown to be false, in general, first by Bousch and Mairesse \cite{BM}, then constructively by Hare, Morris, Sidorov and Theys \cite{HMST}. The finiteness conjecture for rational matrices---equivalent to that for integer matrices---remains open. The methods used for Theorem~\ref{thm:mainrational} establishes a new case of the finiteness conjecture.

\begin{corollary}\label{cor:fc} Let $\mathcal{A}$ be a finite set of $d\times d$ nonnegative matrices with $\rho^*(\mathcal{A})=\rho({\bf A})$. Then $\mathcal{A}$ has the finiteness property.
\end{corollary}

\noindent In fact, in order to prove Theorem \ref{thm:mainrational}, we say much more than the statement in Corollary \ref{cor:fc}. Not only is there one maximal periodic sequence, but all of the maximal sequences are periodic.


\section{Basic properties}\label{sec:prop}

In this section, we collect some basic properties of $k$-regular sequences and their (minimal) linear representations.
First, we give explicit relationships between bases of $V_k(f)$ and linear representations. In particular, we show that one may go between any two minimal degree linear representations by a change of basis. 

\begin{lemma}\label{lem:basrep} Let $f$ be a $k$-regular sequence of degree $d$. Then any $d\times d$ linear representation of $f$ comes from a basis for $V_k(f)$. In particular, if $({\bf u},\mathcal{A},{\bf v})$ is a $d\times d$ linear representation of $f$, then  the $d$ sequences $g_i(n):={\bf u}^T {\bf A}_{(n)_k}{\bf e}_i,$ with $i=1,\ldots,d$ form a basis for $V_k(f)$. 
\end{lemma}

\begin{proof} Let $U$ be the $\mathbb{R}$-vector space generated by $g_1,\ldots,g_d$. So $\dim_{\mathbb{R}}U\leqslant d$. Now, given an $\ell\geqslant 0$ and an $r$ with $0\leqslant r<k^\ell$, we have that $$f(k^\ell n+r)={\bf u}^T {\bf A}_{(k^\ell n+r)_k}{\bf v}={\bf u}^T {\bf A}_{(n)_k}{\bf A}_0^{\ell-{\rm len}(r)}{\bf A}_{(r)_k}{\bf v}\in U,$$ where ${\rm len}(r)$ is the length of $r$ in the base $k$. Thus, every element of ${\rm ker}_k(f)$ is an element of $U$. So $V_k(f)$ is a subspace of $U$. But $d=\dim_\mathbb{R}V_k(f)\leqslant \dim_\mathbb{R}U\leqslant d.$ So $\dim_\mathbb{R}U=d$. Thus $V_k(f)=U$ and $g_1,\ldots,g_d$ form a basis for $V_k(f)$.
\end{proof}

\begin{lemma}
\label{LEM:u-cyclic-span}
Let $({\bf u},\mathcal{A},{\bf v})$ be a linear representation of minimal degree $d$. Then $\{{\bf u}^T {\bf A}_{(n)_k} \}_{n \in \N_0}$ spans $\R^d$.
\end{lemma}

\begin{proof}
By Lemma~\ref{lem:basrep}, the sequences $g_i$ with $g_i(n) = {\bf u}^T {\bf A}_{(n)_k} {\bf e}_i$ are linearly independent.
If $\{{\bf u}^T {\bf A}_{(n)_k} \}_{n \in \N_0}$ do not span the whole space, there exists a vector ${\bf w} \in \R^d$ such that ${\bf u}^T {\bf A}_{(n)_k} {\bf w} = 0$ for all $n \in \N_0$. But this would also imply that $(g_1,\ldots, g_d){\bf w} = 0$ in contradiction to the linear independence of the $g_i$. 
\end{proof}

\begin{lemma}
\label{LEM:minimal-basis-chage}
Up to a change of basis, there is only one minimal linear representation of a regular sequences. That is, if $({\bf u}_a,\mathcal{A},{\bf v}_a)$ and $({\bf u}_c,\mathcal{C},{\bf v}_c)$ are two $d\times d$ linear representations of a $k$-regular sequence $f$, then there is an invertible matrix ${\bf M}$ such that ${\bf u}_a^T {\bf M} = {\bf u}_c^T $, ${\bf M}^{-1} {\bf v}_a = {\bf v}_c $ and ${\bf M}^{-1}{\bf A}_i{\bf M}={\bf C}_i$ for $i\in\{0,\ldots,k-1\}$.
\end{lemma}

\begin{proof}
For each $i\in\{1,\ldots,d\},$ set $a_i:={\bf u}_a^T {\bf A}_{(n)_k}{\bf e}_i$ and $c_i:={\bf u}_c^T {\bf C}_{(n)_k}{\bf e}_i$. Applying Lemma~\ref{lem:basrep}, each of $\{a_1,\ldots,a_d\}$ and $\{c_1,\ldots,c_d\}$ are bases for $V_k(f)$, so there is an invertible matrix ${\bf M}$ such that $(a_1(n),\ldots,a_d(n)){\bf M}=(c_1(n),\ldots,c_d(n))$ for all $n$. That is, for each $n$, 
$
{\bf u}_a^T {\bf A}_{(n)_k}{\bf M}={\bf u}_c^T {\bf C}_{(n)_k},
$
and in particular  ${\bf u}_a^T {\bf M} = {\bf u}_c^T $.
For every $i\in\{0,1,\ldots,k-1\}$ and $n \in \N_0$ we obtain
\[
{\bf u}_a^T{\bf A}_{(n)_k}{\bf M}({\bf M}^{-1}{\bf A}_i {\bf M})={\bf u}_a^T{\bf A}_{(kn+i)_k} {\bf M}={\bf u}_c^T{\bf C}_{(kn+i)_k}={\bf u}_c^T{\bf C}_{(n)_k}{\bf C}_i,
\]
 so that ${\bf u}_c^T{\bf C}_{(n)_k}({\bf M}^{-1}{\bf A}_i {\bf M}-{\bf C}_i)=0.$  Since the degree $d$ is minimal, we observe that ${\rm span}_{\mathbb{R}}\{{\bf u}_c^T{\bf C}_{(n)_k}\}=\mathbb{R}^{1\times d}$ and it follows that ${\bf M}^{-1}{\bf A}_i {\bf M}={\bf C}_i$ holds on the whole space $\mathbb{R}^{1\times d}$. Similarly, ${\bf u}_c^T {\bf C}_{(n)_k} {\bf v}_c = {\bf u}_a^T {\bf A}_{(n)_k} {\bf v}_a = {\bf u}_c^T {\bf C}_{(n)_k} {\bf M}^{-1} {\bf v}_a$ for every $n \in \N_0$ implies ${\bf M}^{-1} {\bf v}_a = {\bf v}_c $.
\end{proof}

\begin{lemma}
\label{LEM:v-cyclic-span}
Let $({\bf u},\mathcal{A},{\bf v})$ be a linear representation of minimal degree $d$. Then $\{ {\bf A}_0^m {\bf A}_{(n)_k} {\bf v} \}_{n,m \in \N_0}$ spans $\R^d$.
\end{lemma}

\begin{proof}
Up to a change of basis, we can assume that the linear representation is canonical, that is, constructed from $\{f_1,\ldots,f_d\} \subset \ker_k(f)$. This means, for each $i \in \{1,\ldots,d\}$, there exist $\ell_i \in \N_0$ and $0 \leqslant r_i < k^{\ell_i}$ such that 
\[
{\bf u}^T {\bf A}_{(n)_k} {\bf e}_i
= f_i(n) = f(k^{\ell_i} n + r_i) 
= {\bf u}^T {\bf A}_{(n)_k} {\bf A}_0^{\ell_i - {\rm len}(r_i)} {\bf A}_{(r_i)_k} {\bf v}.
\]
Since $\{{\bf u}^T {\bf A}_{(n)_k} \}_{n \in \N_0}$ contains a basis of $\R^d$, it follows that
\[
{\bf e}_i = {\bf A}_0^{\ell_i - {\rm len}(r_i)} {\bf A}_{(r_i)_k} {\bf v} \in \{ {\bf A}_0^m {\bf A}_{(n)_k} {\bf v} \}_{n,m \in \N_0}.
\]
Since this holds for each of the Euclidean basis vectors ${\bf e}_i$, the claim follows.
\end{proof}

\section{Group parametrisation}\label{SEC:param}

Let a linear representation $({\bf u}, \mathcal{A}, {\bf v})$ of a nonnegative $k$-regular sequence $f$ be given.
The approximant measures $\mu_n$ can be conveniently expressed via their values on the set of intervals
\begin{equation}\label{eq:Ik}\mathcal{I}_k:=\left\{I_{\ell,m}=\left[\frac{m-k^\ell}{k^{\ell}(k-1)},\frac{m+1-k^\ell}{k^{\ell}(k-1)}\right):\ell\geqslant 0, k^\ell\leqslant m<k^{\ell+1}\right\}.
\end{equation}
More precisely, whenever $n \geqslant \ell$,
\begin{equation}
\label{EQ:mu-on-intervals}
\mu_n(I_{\ell,m}) = \frac{ {\bf u}^T {\bf A}_{(m)_k} {\bf A}^{n-\ell} {\bf v} }{\Sigma_f(n)}
\end{equation}
follows readily from the definition. The normalising factor satisfies the relation 
\begin{equation*}
\Sigma_f(n) = {\bf u}^T ({\bf A} - {\bf A}_0) {\bf A}^n {\bf v}.
\end{equation*} 

Before we discuss the possible limit points of the ghost measure approximants $\mu_n$ in full generality, let us start off with a few examples that indicate the kind of phenomena that we can expect to find.

\begin{example}
Consider the linear representation $({\bf u},\mathcal{A}, {\bf v})$ of a $2$-regular sequence $f$, with 
\[
{\bf A}_0 = \begin{pmatrix}
0 & 2
\\1 & 0
\end{pmatrix},
\quad
{\bf A}_1 = \begin{pmatrix}
0 & 1
\\2 & 0
\end{pmatrix},
\quad
{\bf u} = {\bf e}_1 + {\bf e}_2,
\quad
{\bf v} = {\bf e}_1.
\]
Note that 
\[
{\bf A}^n {\bf v} = \begin{cases}
3^n {\bf e}_1& \mbox{if } n \mbox{ even},
\\3^n {\bf e}_2& \mbox{if } n \mbox{ odd}.
\end{cases}
\]
Hence, we get
\[
\frac{\mu_n([0,1/2))}{\mu_n([1/2,1))}
= \frac{{\bf u}^T {\bf A}_1 {\bf A}_0 {\bf A}^{n-1} {\bf v}}{{\bf u}^T {\bf A}_1 {\bf A}_1 {\bf A}^{n-1} {\bf v}}
= \begin{cases}
2 & \mbox{if } n \mbox{ even},
\\ \frac{1}{2} & \mbox{if }n \mbox{ odd}.
\end{cases}
\]
It is not difficult to see that all accumulation points of $(\mu_n)_{n \in \N}$ are continuous.
It follows that the sequence of measures $(\mu_n)_{n \in \N}$ does not converge weakly.
However, we will see that $(\mu_n)_{n \in \N}$ converges to a $2$-cycle of measures. \exend
\end{example}

\begin{example}
Let $R_{\alpha}$ be the rotation matrix in $\R^2$ with angle $2 \pi \alpha$ such that $\alpha$ is irrational. We consider the $4\times4$ linear representation $({\bf u},\mathcal{A}, {\bf v})$ of a $2$-regular sequence $f$ with 
\[
{\bf A}_0 = 
\left(\begin{array}{@{}c|c@{}}
 2 R_{\alpha}
  & {} \\
\hline
  {} &
  \begin{matrix}
  1 & 0 \\
  0 & 2
  \end{matrix}
\end{array}\right),
\quad
{\bf A}_1 = 
\left(\begin{array}{@{}c|c@{}}
  R_{\alpha}
  & {} \\
\hline
  {} &
  \begin{matrix}
  2 & 0 \\
  0 & 1
  \end{matrix}
\end{array}\right),
\quad {\bf u} = \begin{pmatrix}
1\\
0\\
1\\
1
\end{pmatrix},
\quad {\bf v} = \begin{pmatrix}
1\\
0\\
c_1\\
c_2
\end{pmatrix},
\]
where empty spaces denote $0$ entries, and $c_1,c_2 > 1$. A direct calculation yields
\[
\Sigma_f(n) = {\bf u}^T {\bf A}_1 {\bf A}^n {\bf v} = 3^{n} (2c_1 + c_2 +  \cos(2 \pi (n+1) \alpha)) > 0. 
\]
By a similar calculation, we obtain that $f(n) > 0$ for all $n \in \N_0$.
Further, we get
\[
\mu_n([1/2,1)) = \frac{{\bf u}^T {\bf A}_1^2 {\bf A}^{n-1} {\bf v}}{\Sigma_f(n)}
= \frac{4c_1 + c_2 + \cos(2 \pi (n+1) \alpha)}{3(2c_1 + c_2 + \cos(2 \pi (n+1) \alpha)) },
\]
which converges to an infinite `cycle' of solutions due to the irrationality of $\alpha$. \exend
\end{example}

In order to gain a more systematic understanding of ghost measures, we need to investigate the asymptotic behaviour of the expression in \eqref{EQ:mu-on-intervals}.
To this end, it is instrumental to understand the (properly normalised) limiting behaviour of $({\bf A}^n {\bf v})_{n \in \N}$. 
Note however, that the scalar product with ${\bf u}^T ({\bf A} - {\bf A}_0)$ and ${\bf u}^T {\bf A}_{(m)_k}$ with $m \in \N$ might project to a smaller subspace. It will therefore be convenient to consider the (left) $\mc A$-invariant subspace
\[
V^T = {\rm span} \{{\bf u}^T {\bf A}_{(n)_k} \}_{n \in \N}.
\]
Writing ${\bf P}_V$ for the orthogonal projection to $V$, we will work with the restrictions
\[
\widetilde{\bf A} = {\bf P}_V {\bf A} {\bf P}_V,
\quad
\widetilde{\bf A}_i = {\bf P}_V {\bf A}_i {\bf P}_V,
\]
for all $0 \leqslant i \leqslant k-1$. From the invariance of $V$ it easily follows that for all $\widetilde{\bf w} \in V$ we have
\begin{equation}
\label{EQ:tilde-replacement}
\widetilde{\bf w}^T {\bf A}_i 
= \widetilde{\bf w}^T \widetilde{\bf A}_i,
\end{equation}
for all $0 \leqslant i \leqslant k-1$, and likewise for $\bf A$.
We emphasise that in contrast to Lemma~\ref{LEM:u-cyclic-span} we require $n \neq 0$ in the definition of $V$. That is, we do not necessarily have ${\bf u}^T \in V^T$, but we have ${\bf u}^T {\bf A}_{i} \in V^T$ for all $i \neq 0$ by construction. Hence, whenever $i_1 \neq 0$, we have
\begin{equation}
\label{EQ:tilde-replacement-several}
{\bf u}^T {\bf A}_{i_1} {\bf A}_{i_2} \cdots {\bf A}_{i_n} {\bf v}
= {\bf u}^T {\bf A}_{i_1} \widetilde{\bf A}_{i_2} \cdots \widetilde{\bf A}_{i_n} {\bf v},
\end{equation}
implying that we can replace all matrices except the first by the corresponding restriction.

The vector space $V$ may be strictly smaller than $\R^d$ as illustrated by the following example.

\begin{example}
Let $f = 21^{\infty}$ and $f_2 = 1^{\infty}$ such that ${\rm ker}_2(f) = \{f,f_2 \}$.  A corresponding linear representation is given via ${\bf u} = (2,1)^T$, ${\bf v} = (1,0)^T$ and 
\[
{\bf A}_0 = \begin{pmatrix}
1 & 0\\
0 & 1
\end{pmatrix},
\quad
{\bf A}_1 = \begin{pmatrix}
0 & 0\\
1 & 1
\end{pmatrix}.
\]
Observe that ${\bf u}^T {\bf A}_1 = (1,1)$ and hence ${\bf u}^T {\bf A}_{(n)_k} = (1,1)$ for all $n \geqslant 1$, which spans a one-dimensional subspace of $\R^2$. \exend
\end{example}

For practical purposes, it is worth noticing that $V$, although defined as the span of an infinite set, can in fact be obtained from a finite set. Setting $\mc A' = \{ {\bf A}_1,\ldots, {\bf A}_{k-1} \}$, we can rewrite
\[
\{ {\bf u}^T {\bf A}_{(n)_k} \}_{n \in \N} = \bigcup_{m \geqslant 0} U_m,
\quad U_m = {\bf u}^T \mc A' \mc A^m.
\]
We thereby get $V$ as a limit of a nested sequence of vector spaces $(V_n)_{n \in \N_0}$, with
\[
V_n = {\rm span}(W_n), \quad W_n = \bigcup_{m = 0}^{n} U_m.
\]
Clearly, $V_j \subset V_k \subset V$ as long as $j \leqslant k$. As it turns out, the subset relation can be strict only for a very limited set of indices.
\begin{lemma}\label{LEM:V-stabilise}
If $V_n = V_{n+1}$ for some $n \in \N_0$, then $V= V_n$. In particular, $V = V_d$ where $d$ denotes the dimension of the representation.
\end{lemma}

\begin{proof}
We first show that $V_n = V_{n+1}$ also implies $V_{n+1} = V_{n+2}$. Indeed, $V_n = V_{n+1}$ requires that $W_n$ spans the same space as $W_n \cup U_{n+1}$ which is equivalent to
\[
U_{n+1} \subset {\rm span}(W_n),
\]
implying also that
\[
U_{n+2} = U_{n+1} \mc A \subset {\rm span}(W_n \mc A)
\subset {\rm span}(W_{n+1}).
\]
This in turn yields that $V_{n+1} = V_{n+2}$. By induction, it follows that $V_n = V_k$ for all $k \geqslant n$ and thereby $V= V_n$.
Finally note that the dimension of the vector spaces $V_n$ is strictly increasing in $n$ until $V_n = V_{n+1}$. Since all the vector spaces are embedded in $\R^d$, this happens after at most $d$ steps.
\end{proof}

We consider the (complex) Jordan normal form ${\bf J}$ of $\widetilde{\bf A}$. 
Let $\{ {\bf v}_{\lambda,i,j} \}$ be a normalised Jordan basis, where $\lambda$ refers to the eigenvalue, $i$ enumerates the Jordan blocks and $j$ the generalised eigenvectors in the corresponding Jordan chain. Also, let $\{{\bf u}_{\lambda,i,j} \}$ be the corresponding dual basis with respect to the standard scalar product.
For an element $\lambda$ of the peripheral spectrum $\sigma_p$ (that is, an eigenvalue $\lambda$ with $|\lambda| = \rho(\widetilde{\bf A})$), we consider one of the corresponding Jordan blocks
$
{\bf J}_{\lambda}
.
$
We make use of the following elementary useful result in linear algebra.
\begin{lemma}
\label{LEM:Jordan-block-limit}
Let $r$ be the dimension of a Jordan block ${\bf J}_{\lambda}$ and assume $\lambda \neq 0$. Then,
\[
\lim_{n \to \infty} \frac{ (r-1)! }{n^{r-1} \lambda^{n-r+1} } {\bf J}_{\lambda}^n= {\bf e}_1 \otimes{\bf e}_{r}^T.
\]
\end{lemma}
\begin{proof}[Sketch of proof]
By elementary functional calculus, we have that the leading contribution of ${\bf J}_{\lambda}^n$ comes from the entry in the top right corner, given by
\[
C_n = { n \choose r-1 } \lambda^{n-r+1} \sim \frac{n^{r-1}}{(r-1)!} \lambda^{n-r+1},
\]
and all other entries are in $o(C_n)$. This implies the assertion.
\end{proof}
Given $\lambda$ in the peripheral spectrum $\sigma_p$ let $\{{\bf J}_{\lambda,i} \}_i$ be the corresponding Jordan blocks. For the projective limiting behaviour of $({\bf J}^n)_{n \in \N}$ it suffices to consider those Jordan blocks with the maximal dimension
\[
r = r(\widetilde{\bf A}) = \max \{ \dim {\bf J}_{\lambda,i} \}_{\lambda \in \sigma_p,i}.
\]
Let $I_{\lambda} = \{i : \dim {\bf J}_{\lambda,i} = r \}$ be the indices of those blocks with maximal dimension $r$.
In order to account for the phase difference of eigenvalues on the peripheral spectrum, we write $\lambda_{\varphi} = \me^{{ \rm i} \varphi} \rho$, with $\rho = \rho(\widetilde{\bf A})$. The phase rotation operator ${\bf R}$ is the linear extension of
\[
{\bf R} \colon {\bf v}_{\lambda,i,j} \mapsto \begin{cases}
\me^{- \mi \varphi} {\bf v}_{\lambda,i,j} & \mbox{if } \lambda= \lambda_{\varphi} \in \sigma_p,
\\{\bf v}_{\lambda,i,j} & \mbox{otherwise}.
\end{cases}
\]
Sometimes, it is convenient to identify $\bf R$ with its extension to $\R^d$, given by declaring it to be the identity on $V^{\perp}$.
Note that ${\bf R}$ commutes with $\widetilde{\bf A}$ because it is diagonal in the Jordan basis.
With this notation we can extend Lemma~\ref{LEM:Jordan-block-limit} to the complete matrix $\widetilde{\bf A}$.
\begin{lemma}
\label{LEM:RA-limit}
Given $\rho = \rho(\widetilde{\bf A})$ and 
\[
c_n = \frac{(r-1)!}{n^{r-1} \rho^{n-r+1}},
\]
the iterates $\widetilde{\bf A}^n$ satisfy
\[
\lim_{n \to \infty} c_n {\bf R}^n \widetilde{\bf A}^n = \sum_{\lambda \in \sigma_p, i \in I_{\lambda}} {\bf v}_{\lambda,i,1} \otimes {\bf u}_{\lambda,i,r}^T =: {\bf P},
\]
and ${\bf P}$ commutes with ${\bf R}$ and $\widetilde{\bf A}$.
\end{lemma}
\begin{proof}
Note that the matrix ${\bf R}\widetilde{\bf A}$ has a unique eigenvalue on the peripheral spectrum, given by $\rho$. We also observe that ${\bf R} \widetilde{ \bf A}$ and $\widetilde{\bf A}$ have the same Jordan basis with the same eigenvalue $\lambda$ as long as $\lambda \notin \sigma_{p}$. Considering the Jordan normal form of ${\bf R} \widetilde{ \bf A}$ the convergence follows by an application of Lemma~\ref{LEM:Jordan-block-limit}. The commutation relations can be explicitly verified via matrix multiplication in the Jordan basis of $\widetilde{\bf A}$.
\end{proof}

\begin{remark}
If we work with the real instead of the complex Jordan normal form, we observe that ${\bf R}$ corresponds to a concatenation of (commuting) rotations on two-dimensional subspaces. In particular, ${\bf R}$, ${\bf R} \widetilde{ \bf A}$ and ${\bf P}$ are real matrices with respect to the standard Euclidean basis. \exend
\end{remark}

Let $\gs_p=\{ g_1\rho,\ldots, g_s \rho \}$ be the peripheral spectrum. Note that the element $g = (g_1,\ldots, g_s)$ generates a subgroup of the $s$-dimensional torus, modeled by $S^s$, with $S$ the complex unit circle. More precisely, for $h = (h_1,\ldots,h_s) \in S^s$ we set
\[
gh = (g_1 h_1, \ldots, g_s h_s),
\]
and let $G$ be the subgroup generated by $g$,
\[
G = \overline{ \{ g^n : n \in \N\}}.
\]
The recurrence behaviour of $({\bf R}^n)_{n \in \N}$ is then modeled by the recurrence of $\{g^n \}_{n \in \N}$. More precisely, we obtain a group representation from $G$ onto the space $U(\mathbb{R}^d)$ of unitary matrices via
\[
h \mapsto {\bf R}_h \colon {\bf v}_{\lambda,i,j} \mapsto \begin{cases}
h_m {\bf v}_{\lambda,i,j} & \mbox{if } \lambda= g_m \rho,
\\{\bf v}_{\lambda,i,j} & \mbox{otherwise},
\end{cases}
\]
for $h = (h_1,\ldots,h_s)$.
In particular, $ {\bf R} = {\bf R}_{g^{-1}}$.
\begin{lemma}
\label{LEM:A-limit-cycle}
With the notation of Lemma~\ref{LEM:RA-limit}, the accumulation points of $(c_n \widetilde{\bf A}^n {\bf v})_{ n\in \N}$ are precisely
\[
V_{\lim} = \{ {\bf R}_h {\bf P} {\bf v} \}_{h \in G}.
\]
\end{lemma}
\begin{proof}
Due to Lemma~\ref{LEM:RA-limit}, we have
\[
c_{n} \widetilde{\bf A}^{n} {\bf v}
= {\bf R}^{-n} c_{n} {\bf R}^{n} \widetilde{\bf A}^{n} {\bf v} \sim {\bf R}^{-n} {\bf P} {\bf v},
\]
and hence, the accumulation points of this sequence are  of the form ${\bf Q} {\bf Pv}$ with ${\bf Q}$ an accumulation point of $({\bf R}^{-n})_{n \in \N} = ({\bf R}_{g^{n}})_{n \in \N}$. Since $g$ generates the group $G$, these are precisely $\{ {\bf R}_h\}_{h \in G}$.
\end{proof}

\begin{remark}
In the special case that $\rho$ is the unique eigenvalue of $\widetilde{\bf A}$, the matrix $\bf R$ is the identity and the group $G$ is trivial. In this case, $V_{\lim } = \{ {\bf Pv}\}$ consists of a single eigenvector of $\widetilde{\bf A}$. \exend
\end{remark}

In order to understand the approximants $\mu_n$ of the ghost measures, we would like to replace the scaling provided by $(c_n)_{n\in \N}$ with the one given by $(\Sigma_f(n))_{n \in \N}$. Note that
\[
0 \leqslant \liminf_{n \to \infty} c_n \Sigma_f(n)
 = \liminf_{n \to \infty} c_n {\bf u}^T ({\bf A} - {\bf A}_0) \widetilde{\bf A}^n {\bf v}  
 = \min_{h \in G} {\bf u}^T ({\bf A} - {\bf A}_0) {\bf R}_h {\bf P v},
\]
which may or may not be strictly positive. We therefore introduce the following concept.

\begin{definition}\label{def:nondegenerate}
A linear representation $({\bf u}, \mathcal{A},{\bf v})$ is called \emph{nondegenerate} if it satisfies
\[
{\bf u}^T ({\bf A} - {\bf A}_0) {\bf v}' > 0
\]
for all ${\bf v}' \in V_{\lim}$.
We say that a $k$-regular sequence $f$ is nondegenerate, if it has a minimal nondegenerate linear representation.
\end{definition}

Note that by the compactness of $V_{\lim}$, nondegeneracy of $f$ in fact implies that ${\bf u}^T({\bf A} - {\bf A}_0) {\bf v}'$ is \emph{uniformly} bounded away from $0$ for ${\bf v}' \in V_{\lim}$.
By the observations above, this is equivalent to the statement that $(c_n \Sigma_f(n))_{n \in \N}$ is (eventually) bounded away from $0$, which is easy to check via standard linear algebra.
In the special case that $\rho(\widetilde{\bf A})$ is the unique eigenvalue of maximal modulus, this is not much of a restriction.
\begin{lemma}
Assume we have chosen a minimal representation of $f$, satisfying $\widetilde{\bf A} \neq 0$ and that $\rho = \rho(\widetilde{\bf A})$ is the unique eigenvalue on the peripheral spectrum of $\widetilde{\bf A}$. Then, the corresponding representation is nondegenerate.
\end{lemma}

\begin{proof}
First note ${\bf v}' = {\bf P v}$ is unique under these assumptions. For all $n \in \N$ and $r,s \in \N_0$, we have
\[
{\bf u}^T {\bf A}_{(n)_k} {\bf P} {\bf A}^r_0 {\bf A}_{(s)_k} {\bf v}
= \lim_{n \to \infty} c_n {\bf u}^T {\bf A}_{(n)_k} {\bf A}^n {\bf A}^r_0 {\bf A}_{(s)_k} {\bf v} \geqslant 0
\]
because every element on the right hand side is a product of $c_n$ with a sum of elements in $f$. By construction, ${\bf P}$ leaves $V$ invariant and $\widetilde{\bf A} \neq 0$ implies also that ${\bf P} \neq 0$. Since $V$ is spanned by $\{ {\bf u}^T {\bf A}_{(n)_k} \}_{n \in \N}$, there is an $n \in \N$ such that ${\bf u}^T {\bf A}_{(n)_k} {\bf P} \neq 0$. By Lemma~\ref{LEM:v-cyclic-span}, there are also $r,s \in \N_0$ such that the scalar product with ${\bf A}^r_0 {\bf A}_{(s)_k} {\bf v}$ does not vanish, and hence,
\[
{\bf u}^T {\bf A}_{(n)_k} {\bf P} {\bf A}^r_0 {\bf A}_{(s)_k} {\bf v} > 0.
\]
Taking sums over all matrix products with the same length, we get
\[
0< {\bf u}^T ({\bf A} - {\bf A}_0) {\bf A}^{\text{len}(n) - 1} {\bf P} {\bf A}^{r + {\rm len}(s)} {\bf v}
= \rho^{{\rm len}(n) + {\rm len}(s) + r -1} {\bf u}^T ({\bf A} - {\bf A}_0) {\bf v}',
\]
which implies the desired assertion.
\end{proof}

\begin{remark}
In general, nondegeneracy cannot be replaced by the weaker assumption $\widetilde{\bf A} \neq 0$. Consider for example a $2$-regular sequence of degree $2$, with linear representation ${\bf u} = {\bf v} = (1,1)$ and ${\bf A}_0 = {\bf A}_1 = {\rm diag}(1,-1)$. Then, 
\[
\Sigma_f(N) = \begin{cases}
0 & \mbox{if } N \mbox{ even},
\\2^N & \mbox{if } N \mbox{ odd}.
\end{cases}
\]
and so $\mu_N$ is not well defined for $N \in 2\N$. \exend
\end{remark}

\begin{lemma}
\label{LEM:weak-convergence}
Assume that $(X,d)$ is a compact metric space and $(\mu_n)_{n \in \N}$ is a sequence of Borel probability measures on $X$. Suppose that there is a sequence $(\xi_n)_{n \in \N}$ of partitions of $X$ with the following properties
\begin{itemize}
\item $ \lim_{n \to \infty} \sup_{A \in \xi_n} \operatorname{diam} (A) = 0$, that is, the diameter of the partitions converges to $0$, and
\item $(\mu_m(A))_{m \in \N}$ converges for all $A \in \xi_n$ and $n \in \N$ to some value $\overline{\mu}(A)$.
\end{itemize}
Then, the sequence $(\mu_n)_{n \in \N}$ converges in the weak topology to some Borel probability measure $\mu$. Further, for each $n \in \N$, let $\mc F_n$ be the functional
\[
\mc F_n \colon g \mapsto \sum_{A \in \xi_n} \overline{\mu}(A) \sup_A g,
\]
for every continuous function $g \colon X \to \R$. Then, for every $n \in \N$,
\[
| \mc F_{n}(g) - \mu(g)| \leqslant {\rm var}_{n} g 
:= \sup_{A \in \xi_n} \biggl(\sup_A g - \inf_A g \biggr).
\]
\end{lemma}

\begin{proof}
Since $(\mu_n)_{n \in \N}$ has at least one weak accumulation point, convergence follows if $\mu_n(g)$ is a Cauchy sequence for every continuous function $g \colon X \to \R$. Due to the compactness of $X$, each such $g$ is in fact uniformly continuous, implying that ${\rm var}_n g \to 0$ as $n \to \infty$. 
For each $n,m \in \N$, consider
\[
|\mc F_m(g) - \mu_n(g)| 
\leqslant \sum_{A \in \xi_m} \Bigl|\overline{\mu}(A) \sup_A g - \mu_n(\mathbf{1}_A g) \Bigr|.
\]
We further estimate
\begin{align*}
\overline{\mu}(A) \sup_A g - \mu_n(\mathbf{1}_A g)
& \leqslant \overline{\mu}(A) \sup_A g - \mu_n(A) \inf_A g
\\& \leqslant (\overline{\mu}(A) - \mu_n(A)) \sup_A g + \mu_n(A) {\rm var}_m g
\\& \xrightarrow{n \to \infty} \overline{\mu}(A) {\rm var}_m g,
\end{align*}
and we get the same bound for the absolute value by a similar calculation.
It follows that
\[
\limsup_{n \to \infty} | \mc F_m(g)- \mu_n(g)| \leqslant \sum_{A \in \xi_m} \overline{\mu}(A) {\rm var}_m g = {\rm var}_m g.
\]
By the triangle inequality, all accumulation points of $(\mu_n(g))_{n\in \N}$ have distance at most $2\,{\rm var}_m g$. Since ${\rm var}_m g$ can be chosen arbitrarily small, it follows that $(\mu_n(g))_{n\in \N}$ is Cauchy and the weak convergence follows. In particular,
\[
| \mc F_m(g) - \mu(g) | = \lim_{n \to \infty} |\mc F_m(g) - \mu_n(g)| \leqslant {\rm var}_m g,
\]
which proves the second claim.
\end{proof}

\begin{remark}
\label{REM:limit-measure-boundary}
Under the assumptions of Lemma~\ref{LEM:weak-convergence}, the limiting measure $\mu$ does not necessarily satisfy the relation $\mu(A) = \overline{\mu}(A)$ for all $A \in \xi_n$ and $n \in \N$. This is because $\mu$ might assign mass to the boundary of $A$. However, the relation holds as soon as the boundary of $A$ is a finite set and $\mu$ a continuous measure. Also, it follows from the second assertion in Lemma~\ref{LEM:weak-convergence} that $\mu$ is completely determined by the values $\overline{\mu}(A)$ with $A \in \xi_n$ and $n \in \N$. \exend
\end{remark}

\begin{proof}[Proof of Theorem \ref{THM:ghost-measure-group}] Let $({\bf u}, \mathcal{A},{\bf v})$ be a nondegenerate linear representation of $f$.
The ghost measure approximants of $f$ satisfy for $n \geqslant \ell$,
\begin{align*}
\mu_n(I_{\ell,m}) & = \frac{{\bf u}^T {\bf A}_{(m)_k} {\bf A}^{n-\ell} {\bf v}}{{\bf u}^T ({\bf A} - {\bf A}_0) {\bf A}^n {\bf v}}
= \frac{c_n}{c_{n-\ell}}\cdot \frac{{\bf u}^T {\bf A}_{(m)_k} {\bf R}^{\ell -n} c_{n-\ell}({\bf AR})^{n-\ell} {\bf v} }{{\bf u}^T ({\bf A} - {\bf A}_0) {\bf R}^{-n} c_n ({\bf AR})^n {\bf v}}
\\ & \sim \rho^{-\ell} \frac{{\bf u}^T {\bf A}_{(m)_k} {\bf R}^{\ell} {\bf R}^{-n} {\bf Pv}}{{\bf u}^T ({\bf A} - {\bf A}_0) {\bf R}^{-n} {\bf Pv}}
= C_m(g^{n}),
\end{align*}
where, for $k^{\ell} \leqslant m < k^{\ell+1}$ and $h \in G$,
\[
C_m(h) := \rho^{-\ell} \frac{{\bf u}^T {\bf A}_{(m)_k} {\bf R}^{\ell} {\bf R}_h {\bf Pv}}{{\bf u}^T ({\bf A} - {\bf A}_0) {\bf R}_h {\bf Pv}}
\]
defines a continuous function $h \mapsto C_m(h)$ because the denominator is non-vanishing by the nondegeneracy assumption.
Hence, whenever $\lim_{j \to \infty} g^{n_j} = h$, we get that $\lim_{j\to \infty}\mu_{n_j}(I_{\ell,m}) = C_m(h)$ for all $m$ and the sequence of measures $(\mu_{n_j})_{j}$ converges weakly to some measure $\mu_h$ by Lemma~\ref{LEM:weak-convergence}. 
Conversely, if $(\mu_{n_j})_{j \in \N}$ converges weakly, let $h$ be an accumulation point of $(g^{n_j})_{j \in \N}$. By the discussion above, the corresponding subsequence of $(\mu_{n_j})_{j \in \N}$ converges to $\mu_h$ and since the limit is unique, we get that $\lim_{j \to \infty} \mu_{n_j} = \mu_h$.
Setting $\Psi_f(h) : = \mu_h$, this shows that the accumulation points of $(\mu_n)_{n \in \N}$ are precisely
\[
\{ \mu_h \}_{h \in G} = \Psi_f(G).
\]
With this convention, the last statement of the theorem follows by the considerations above. It remains to show the continuity of $\Psi_f$. To this end, we make use of the second statement in Lemma~\ref{LEM:weak-convergence}. For each $h \in G$, and continuous function $F \colon \TT \to \R$, let
\[
\mc F_{h,\ell}(F) = \sum_{m = k^{\ell}}^{k^{\ell+1}-1} C_m(h) \sup_{I_{\ell,m}} F,
\]
and note that Lemma~\ref{LEM:weak-convergence} implies
\[
|\mu_h(F) - \mc F_{h,\ell}(F)| \leqslant {\rm var}_{\ell} F = \sup_{k^{\ell} \leqslant m < k^{\ell+1} } \Bigl( \sup_{I_{\ell,m}} F - \inf_{I_{\ell,m}} F \Bigr).
\]
In particular, this bound is independent of $h$. Let $(h_n)_{n \in \N}$ be a sequence that converges to $h$ in $G$. Given a continuous function $F$ and $\varepsilon > 0$, let $\ell$ be large enough that ${\rm var}_{\ell} F < \varepsilon$. Since $\mc F_{h,\ell}$ is continuous in $h$, there exists $n_0\in \N$ such that for all $n \geqslant n_0$, we have
\[
|\mc F_{h,\ell}(F) - \mc F_{h_n,\ell}(F)| < \varepsilon.
\]
By a standard $3\varepsilon$-argument, this implies the continuity of $h \mapsto \mu_h(F)$. Since $F$ was arbitrary, we have continuity of $h \mapsto \mu_h$ in the weak topology.
\end{proof}

In addition to Corollary \ref{cor:unique} stated in the Introduction, we have the following.

\begin{corollary}
For every nonnegative and nondegenerate real-valued $k$-regular sequence $f$, the set of ghost measures is compact in the weak topology and has finitely many connected components.
\end{corollary}

\begin{corollary}\label{COR:mu-continuous-on-I} Let
$({\bf u},\mathcal{A},{\bf v})$ be a nondegenerate representation of $f$ and suppose that $\mu_h = \Psi_f(h)$ is a continuous ghost measure of $f$. Then,
\[
\mu_h(I_{\ell,m}) = C_m(h) 
\]
for all $I_{\ell,m} \in \mc I_k$, where
\begin{equation}
\label{EQ:Cmh}
C_m(h) = \rho^{-\ell} \frac{{\bf u}^T {\bf A}_{(m)_k} {\bf R}^{\ell} {\bf R}_h {\bf Pv}}{{\bf u}^T ({\bf A} - {\bf A}_0) {\bf R}_h {\bf Pv}}.
\end{equation}
\end{corollary}

\begin{remark} Corollary \ref{COR:mu-continuous-on-I} can be understood as a level-set procedure for constructing continuous ghost measures. In this way, one could start to make analogies with patch frequency measures related to symbolic dynamical systems.\exend
\end{remark}

\begin{remark}
\label{REM:limit-sequence}
We can find more structure if ${\bf R}$ commutes with ${\bf A}_i$ for all $0 \leqslant i \leqslant k-1$.
This is the case if the Jordan basis for the eigenvalues on the peripheral spectrum of $\widetilde{\bf A}$ coincides for all $\widetilde{\bf A}_i$. Alternatively, it is true, whenever $\rho(\widetilde{\bf A})$ is the unique eigenvalue of maximal modulus (implying that ${\bf R}$ is the identity).
Assuming ${\bf R}$ commutes with all ${\bf A}_i$, given $h \in G$, let $f_h$ be the $k$-regular sequence with representation $(({\bf R}^{T})^{-1} {\bf u} , {\bf R}\cdot\mathcal{A}, {\bf v}_h)$, where
\[
{\bf v}_h = \frac{{\bf R}_h {\bf Pv} }{{\bf u}^T ({\bf A} - {\bf A}_0) {\bf R}_h {\bf Pv}}.
\]
A straightforward calculation yields
\[
\Sigma_{f_h}(n) = \rho^n {\bf u}^T ({\bf A} - {\bf A}_0) {\bf v}_h = \rho^n
\]
and the corresponding ghost measure approximants $(\mu_{h,n})_{n \in \N }$ satisfy
\[
\mu_{h,n} (I_{\ell,m}) = \frac{{\bf u}^T {\bf A}_{(m)_k} {\bf R}^{\ell} {\bf v}_h}{\rho^{\ell}} = \frac{f_h(m)}{\rho^{\ell}} = C_m(h),
\]
for all $n \geqslant \ell$. In particular, $f_h$ has a unique ghost measure, given by $\mu_h$. \exend
\end{remark}

\begin{definition}
We call $({\bf w},\mathcal{B},{\bf x})$ a \emph{reduced representation} of a nonnegative $k$-regular sequence $g$ if it fulfills the following properties:
\begin{enumerate}
\item[(i)] The spectral radius $\rho = \rho(\widetilde{\bf B})$ is the unique maximal eigenvalue of $\widetilde{\bf B}$,
\item[(ii)] ${\bf x}$ is a $\rho$-eigenvector of $\widetilde{\bf B}$,
\item[(iii)] $\Sigma_g(n) = \rho^n$ for all $n \in \N$, and
\item[(iv)] $g$ has a unique ghost measure $\mu_g$,
\end{enumerate}
where ${\bf B} = \sum_{i=0}^{k-1} {\bf B}_i$, and $\widetilde{\bf B}$ is the restriction of $\bf B$ to the subspace generated by $\{{\bf w}^T {\bf B}_{(n)_k} \}_{n \in \N}$.
\end{definition}

With this definition, the following is an immediate consequence of the discussion in Remark~\ref{REM:limit-sequence}.

\begin{corollary}
\label{COR:reduced-form}
Let $\mu_h$ be a ghost measure of a nonnegative nondegenerate $k$-regular sequence $f$. Assume further that ${\bf R}$ commutes with ${\bf A}_i$ for all $0\leqslant i \leqslant k - 1$. Then, there exists a nonnegative $k$-regular sequence $g_h$ with reduced representation $({\bf w},\mathcal{B},{\bf x})$ such that $\mu_h$ is the unique ghost measure of $g_h$.
\end{corollary}


\section{Determining the spectral type}\label{sec:spectype}


In this section, we give several criteria to determine the spectral type of ghost measures. 
Throughout this section, we assume that $f$ is a nonnegative  nondegenerate real-valued $k$-regular sequence with minimal representation given by $({\bf u},\mathcal{A},{\bf v})$.

\subsection{Pure point part of ghost measures}

As discussed in Remark~\ref{REM:limit-measure-boundary}, the equality recorded in Corollary~\ref{COR:mu-continuous-on-I} may fail if the ghost measure $\mu_h$ is not continuous. 
In the following result, we describe the potential Bragg peaks of $\mu_h$.

\begin{proposition}
\label{PROP:pure-points}
Suppose that $f$ is a nonnegative real-valued $k$-regular sequence with linear representation $({\bf u},\mathcal{A},{\bf v})$. Also, given $x\in[0,1)$, let the sequence $(m_\ell)_{\ell\in\N_0}$ be defined by the equality $x=\bigcap_{\ell\in\N_0} I_{\ell,m_\ell}$. Then, for all $h \in G$, 
\begin{itemize}
\item If $x \notin \frac{1}{k-1}\mathbb{Z}[k^{-1}]$, we have
\[
\mu_h(\{x\}) = \lim_{ \ell \to \infty} C_{m_{\ell}}(h).
\]
\item If $x \in \frac{1}{k-1}\mathbb{Z}[k^{-1}]$, we have
\[
\mu_h(\{x\}) = \lim_{ \ell \to \infty}
\bigl( C_{m_{\ell}}(h) + C_{m_{\ell}-1}(h) \bigr),
\]
\end{itemize}
with $C_m(h)$ as in \eqref{EQ:Cmh}.
\end{proposition}

\begin{proof}
Since $\mu_h$ is regular, we get for every $x \in \TT$,
\[
\mu_h(\{x \}) = \lim_{r \to 0^+} \mu_h(B_r(x)),
\]
where $B_r(x)$ is the ball of radius $r$ centred at $x$.
As in the statement of the result, we distinguish two cases. First, let us assume that $x \notin \frac{1}{k-1}\Z[k^{-1}]$. Then, for each $\ell \in \N$, there is an integer $m_\ell\in[k^\ell,k^{\ell+1})\cap\mathbb{N}$ such that $x\in\operatorname{Int}(I_{\ell,m_\ell})$, that is, $x$ is an interior point of the interval $I_{\ell,m_\ell}$. 
Let $(n_j)_{j \in \N}$ be such that $\lim_{j \to \infty} \mu_{n_j} = \mu_h$.  It follows that
\begin{equation*}
\mu_h(\{ x\}) \leqslant \mu_h(\operatorname{Int}(I_{\ell,m_\ell})) \leqslant \liminf_{j \to \infty} \mu_{n_j}(\operatorname{Int}(I_{\ell,m_\ell}))
 \leqslant \lim_{j \to \infty} \mu_{n_j}(I_{\ell,m_\ell}) = C_{m_{\ell}}(h).
\end{equation*}
Hence, we get
\begin{equation}\label{eq:muinf}
\mu_h(\{x\}) \leqslant \liminf_{\ell \to \infty} C_{m_{\ell}}(h).
\end{equation}
On the other hand, writing $\overline{I_{\ell,m_\ell}}$ as the closure of the interval $I_{\ell,m_\ell}$, we obtain
\[
\mu_h\bigl(\overline{I_{\ell,m_{\ell}}} \bigr) \geqslant \lim_{j \to \infty} \mu_{n_j}(I_{\ell,m_{\ell}}) 
= C_{m_{\ell}}(h).
\]
Taking the limit superior in $\ell$ on both sides yields
\[
\mu_h(\{x\}) \geqslant \limsup_{\ell \to \infty} C_{m_{\ell}}(h).
\]
Combining this lower bound with the upper bound in \eqref{eq:muinf} proves the first result of the proposition.

We now turn to the case $x \in \frac{1}{k-1}\Z[k^{-1}]$. 
In this case, there is some minimal $r \in \N$ and some integer $m_r\in[k^r,k^{r+1})$ such that $x$ is the left boundary point of the half-open interval $I_{r,m_r}$ and the right boundary point of the closed interval $\overline{I_{m_r-1}}$, where $I_{m_r -1} = I_{r-1,m_r - 1}$ if $m_r = k^r$, and $I_{m_r -1} = I_{r,m_r-1}$ otherwise.
For all $\ell \geqslant r$, we obtain
\begin{align*}
\mu_h(\{x\}) & \leqslant \mu_h(\operatorname{Int}( I_{\ell,m_{\ell}}\cup I_{m_{\ell}-1}))\\ 
& \leqslant \lim_{j \to \infty}\big( \mu_{n_j}(I_{\ell,m_{\ell}}) + \mu_{n_j}(I_{m_{\ell}-1})\big)
\\ & = C_{m_{\ell}}(h) + C_{m_{\ell} - 1}(h).
\end{align*} 
On the other hand,
\begin{align*}
\mu_h\big(\overline{I_{\ell,m_{\ell}} \cup I_{m_{\ell}-1}} \big)
& \geqslant \lim_{j \to \infty} \big(\mu_{n_j}(I_{\ell,m_{\ell}}) + \mu_{n_j}(I_{m_{\ell}-1})\big)
\\ & = C_{m_{\ell}}(h) + C_{m_{\ell} - 1}(h).
\end{align*}
As above, taking limits of these two inequalities gives the desired result.
\end{proof}

In order to obtain a better understanding of the location $y$ of a Bragg peak, we wish to have a more explicit representation of the product ${\bf A}_{(m_{\ell})_k}$ for the unique integer $m_{\ell}$ such that $y \in I_{\ell,m_{\ell}}$.  
It turns out that this can be tied to an appropriate coding of the point $y$. 
Recall that every point $y \in \TT$, identified with the unit interval $[0,1]$, can be written as
\[
y = \pi(x) : = \sum_{n=1}^{\infty} \frac{x_n - \delta_{n,1}}{(k-1)k^{n-1}},
\]
for some integer valued sequence $ x = (x_n)_{n \in \N}$ in the coding space
\[
\mathbb{X}_k = \{1,\ldots,k-1\} \times \{0,\ldots,k-1 \}^{\N}.
\]
This representation is unique for $y \notin \frac{1}{k-1}\Z[k^{-1}] $ and $2$-to-$1$ otherwise. This coding is useful because it determines the matrix products to be taken in order to calculate ghost measures on small intervals around $x$. More precisely, as a consequence of Proposition~\ref{PROP:pure-points}, we obtain the following.
\begin{corollary}
\label{COR:point-mass-sequence}
Let $y \in \TT$ and $\mu_h$ be a ghost measure on $\TT$. Then,
\[
\mu_h(\{y\}) = \lim_{n \to \infty} c_n(h,y),
\]
where
\[
c_n(h,y) = \sum_{x \in \pi^{-1}(y)} \frac{{\bf u}^T {\bf A}_{x_1} \cdots {\bf A}_{x_{n+1}} {\bf R}^{n} {\bf R}_h {\bf P} {\bf v}}{\rho^n {\bf u}^T ({\bf A} - {\bf A}_0) {\bf R}_h {\bf P} {\bf v}},
\]
for all $n \in \N$.
\end{corollary}

\begin{proof}
Let $(m_{\ell})_{\ell \in \N_0}$ be the unique sequence such that $y \in I_{\ell,m_{\ell}}$ for all $\ell \in \N_0$. By the definition of $I_{\ell,m_{\ell}}$, this is equivalent to
\[
m_{\ell} = \lfloor y (k-1)k^{\ell}  + k^{\ell} \rfloor.
\]
For $x \in \pi^{-1}(y)$, this can be rewritten as
\[
m_{\ell} = \biggl\lfloor \sum_{n = 1}^{\infty}  x_n k^{\ell + 1 - n} \biggr\rfloor.
\]
We distinguish two cases. First, if $y \notin \frac{1}{k-1}\Z[k^{-1}]$, then $x$ is unique and not eventually equal to $k-1$, implying that
\begin{equation}
\label{EQ:m_ell-x_n-relation}
m_{\ell} = \sum_{n=1}^{\ell+1} x_n k^{\ell+1-n},
\end{equation}
and therefore $A_{(m_\ell)_k} = A_{x_1} \cdots A_{x_{\ell+1}}$, which implies the claim by Proposition~\ref{PROP:pure-points}.
On the other hand, if $y \in \frac{1}{k-1} \Z[k^{-1}]$, we obtain two points in $\pi(y)$. One of them satisfies \eqref{EQ:m_ell-x_n-relation} whereas for the other one the sum in \eqref{EQ:m_ell-x_n-relation} is equal to $m_{\ell}-1$. Again, the claim follows by Proposition~\ref{PROP:pure-points}.
\end{proof}

\subsection{Continuous part of ghost measures}\label{SEC:contin}

Here, we provide finer characterisations of continuous ghost measures. In addition to the spectral radius, there are two other important quantities we focus on in this section: the joint spectral radius and the Lyapunov exponent of a finite set of matrices. Recall, the {\em joint spectral radius} of a finite set of matrices $\mathcal{S}=\{{\bf M}_1,{\bf M}_2,\ldots, {\bf M}_{k}\}$ is the real number $$\rho^*=\rho^*(\{{\bf M}_1,{\bf M}_2,\ldots, {\bf M}_{k}\}):=\lim_{n\to\infty}\max_{1\leqslant i_1,i_2,\ldots,i_{n}\leqslant k}\big\| {\bf M}_{i_1}{\bf M}_{i_2}\cdots{\bf M}_{i_{n}}\big\|^{1/n},$$
where $\|\cdot\|$ is any (submultiplicative) matrix norm. This quantity was introduced by Rota and Strang \cite{RS1960} and has a wide range of applications. For an extensive treatment, see Jungers' monograph \cite{J2009}. We denote the {\em Lyapunov exponent} of a finite set of matrices $\{{\bf M}_1,{\bf M}_2,\ldots, {\bf M}_{k}\}$, by the value $\log_k\bar{\rho}$, where $$\bar{\rho}=\bar{\rho}(\{{\bf M}_1,{\bf M}_2,\ldots, {\bf M}_{k}\}):=\lim_{n\to\infty}\left(\prod_{1\leqslant i_1,i_2,\ldots,i_{n}\leqslant k}\big\| {\bf M}_{i_1}{\bf M}_{i_2}\cdots{\bf M}_{i_{n}}\big\|^{\frac{1}{n}}\right)^{1/k^n},$$
where $\|\cdot\|$ is any (submultiplicative) matrix norm. Note that the above limit is decreasing to $\bar{\rho}$. In our setting, the Lyapunov exponent describes $\lambda$-almost everywhere behaviour for $\|{\bf M}_{w_1\cdots w_n}\|$; that is, for $\lambda$-almost all $x\in[0,1)$ with $x=\pi(w)$, with $w=w_1w_2\cdots$, we have $$\bar{\rho}=\lim_{n\to\infty}\big\|{\bf M}_{w_1}{\bf M}_{w_2}\cdots{\bf M}_{w_{n}}\big\|^{1/n},$$ where $\lambda$ is Lebesgue measure. The Lyapunov exponent is the natural extension of the geometric mean of real numbers---in dimension one, they are equal. 

\begin{definition}[Cone] A (linear) cone $K$ is a subset with nonempty interior of a finite dimensional vector space such that $K$ is closed both under linear combinations with nonnegative coefficients and in the standard topology, with $K\cap(-K)=\{0\}.$
\end{definition}

If the matrix semigroup $\mathcal{S}$ is cone-preserving, we have the fundamental inequality
\begin{equation}\label{eq:fund}\bar{\rho}\leqslant\frac{\rho}{k}\leqslant \rho^*\leqslant \rho.\end{equation} Here, $\rho$ is the spectral radius of ${\bf M} = \sum_{i=1}^k {\bf M}_i$. 
The last inequality was given by Blondel and Nesterov \cite[Thm.~1]{BN}, while the middle inequality holds for all finitely-generated matrix semigroups. To show that the first inequality holds, we will use a reverse triangle inequality for cone-preserving matrices. We require the following lemma of Protasov \cite[Lemma 1]{P2000}.

\begin{lemma}[Protasov \cite{P2000}]\label{lem:c} Let $K\subset\mathbb{R}^n$ be a cone and $\|\cdot\|$ be a vector norm. Then there is a $c>0$ depending on the choice of norm, such that for any integer $m\geqslant 0$ and vectors ${\bf y}_1,\ldots,{\bf y}_m\in K$, we have $$\left\|\sum_{i=1}^m{\bf y}_i\right\|\geqslant c\sum_{i=1}^m\left\|{\bf y}_i\right\|.$$
\end{lemma}

Since the set of $n\times n$ matrices that fix a cone $K$ themselves form a cone in the space of $n\times n$ matrices, we have the following corollary of Lemma~\ref{lem:c}.

\begin{corollary}\label{cor:Bcone} Let ${\bf M}_1,\ldots,{\bf M}_{m}$ be $n\times n$ matrices that fix a cone $K$ and $\|\cdot\|$ be a vector-induced matrix norm. Then there is a $c>0$ depending on the choice of norm, such that $$\left\|\sum_{i=1}^m{\bf M}_i\right\|\geqslant c\sum_{i=1}^m\big\|{\bf M}_i\big\|.$$
\end{corollary}

\begin{proposition} If $\mathcal{S}=\{{\bf M}_1,{\bf M}_2,\ldots, {\bf M}_{k}\}$ preserves a cone, then $\bar{\rho}\leqslant\rho/k.$
\end{proposition}

\begin{proof} Note that $$\frac{\rho}{k}=\lim_{n\to\infty}\left(k^{-n}\sum_{1\leqslant i_1,i_2,\ldots,i_{n}\leqslant k}\big\| {\bf M}_{i_1}{\bf M}_{i_2}\cdots{\bf M}_{i_{n}}\big\|\right)^{1/n}.$$ To see this, we use both the classical and reverse triangle inequalities to give \begin{align*} \frac{1}{k}\cdot\|{\bf M}^n\|^{1/n}&=\frac{1}{k}\left\|\sum_{1\leqslant i_1,i_2,\ldots,i_{n}\leqslant k} {\bf M}_{i_1}{\bf M}_{i_2}\cdots{\bf M}_{i_{n}}\right\|^{1/n}\\ &\leqslant \left(k^{-n}\sum_{1\leqslant i_1,i_2,\ldots,i_{n}\leqslant k}\big\| {\bf M}_{i_1}{\bf M}_{i_2}\cdots{\bf M}_{i_{n}}\big\|\right)^{1/n}\\
&\leqslant \frac{c^{1/n}}{k}\left\|\sum_{1\leqslant i_1,i_2,\ldots,i_{n}\leqslant k} {\bf M}_{i_1}{\bf M}_{i_2}\cdots{\bf M}_{i_{n}}\right\|^{1/n}=\frac{c^{1/n}}{k}\cdot\|{\bf M}^n\|^{1/n}.
\end{align*} Taking the limit on each side proves the claim. The proposition now follows noting that $\bar\rho$ is at most this limit by the arithmetic-geometric mean inequality.
\end{proof}

Our results in the remainder of this section consider the consequences of any one of these inequalities holding with equality, or being strict, for our finite set of matrices ${\bf A}_0,\ldots,{\bf A}_{k-1}$ from a minimal representation of $f$ restricted to an appropriate subspace.

By Corollary~\ref{COR:point-mass-sequence}, the existence of Bragg peaks for a ghost measure $\mu_h$ is clearly tied to the growth behaviour of matrix products ${\bf A}_{x_1} \cdots {\bf A}_{x_{n}}$ as compared to the growth of $\rho^n$ as $n \to \infty$. In order to obtain reasonably sharp criteria for the existence (or absence) of eigenvalues, we have to take into account that the action on vectors of the form ${\bf R}_h {\bf P} {\bf v}$ might implicitly project to a smaller subspace. We therefore proceed similarly as in Section~\ref{SEC:param} and define a subspace of $V$, given by
\begin{equation}\label{def:vhat}
\widehat{V} = \operatorname{span} \bigl\{ \widetilde{\bf A}_0^{m} \widetilde{\bf A}_{(n)_k} {\bf R}_h {\bf P} {\bf v} \bigr\}_{m,n \in \N_0,h \in G},
\end{equation}
where, by convention, $\widetilde{\bf A}_{(0)_k}$ denotes the identity, and ${\bf R}_h$ and ${\bf P}$ are as defined in Section \ref{SEC:param}.
 As in Lemma~\ref{LEM:V-stabilise}, one only needs finitely many steps to construct $\widehat{V}$.
We denote the restrictions of $\widetilde{\bf A}_i$ to $\widehat{V}$ by $\widehat{\bf A}_i = {\bf P}_{\widehat{V}} \widetilde{\bf A}_i {\bf P}_{\widehat{V}}$, where ${\bf P}_{\widehat{V}}$ is the orthogonal projection to $\widehat{V}$.
Note that $\widehat{V}$ is invariant under $\widetilde{\bf A}_i$ for each $i$ and therefore, we have for every ${\bf v}'\in \widehat{V}$ and $n \in \N$,
\begin{equation}
\label{EQ:hat-replacement}
\widetilde{\bf A}_{i_1} \cdots \widetilde{\bf A}_{i_n} {\bf v}'
= \widehat{\bf A}_{i_1} \cdots \widehat{\bf A}_{i_n} {\bf v}',
\end{equation}
irrespective of the choices for $i_j$ for all $1\leqslant j \leqslant n$. 

Given a collection of vectors $W$, the positive span of $W$ is given by
\[
\Span_+(W) = \biggl \{ \sum_{i=1}^n a_i {\bf w}_i : n \in \N, a_i \in \R_{\geqslant 0}, {\bf w}_i \in W \mbox{ for all } 1\leqslant i \leqslant n \biggr \}.
\]

\begin{lemma}
\label{LEM:cone-condition}
There is a cone $K \subset \widehat{V}$, fixed by each $\widehat{\bf A}_i$ in the sense that $\widehat{\bf A}_i K  \subset K$.
\end{lemma}

\begin{proof}
Set $U=\bigl\{ \widetilde{\bf A}_0^{m} \widetilde{\bf A}_{(n)_k} {\bf R}_h {\bf P} {\bf v} \bigr\}_{m,n \in \N_0,h \in G}$ and $K = \overline{\Span_+ (U)}.$ Note that $U$ is invariant under the application of $\widehat{\bf A}_i$ by construction and hence the same holds for its positive span (by linearity) and the corresponding closure (by continuity). That is, $K$ is invariant under $\widehat{\bf A}_i$ for all $i$ and it remains to show that $K$ is indeed a cone.
By definition, $K$ is topologically closed and as the closure of a positive span, it is invariant under linear combinations with positive coefficients. 
Since the underlying sequence is nondegenerate, $C_m(h)\geqslant 0$ for all $h\in G$, and it follows that every ${\bf y} \in U$ has a nonnegative scalar product with all elements of $W = \{{\bf u}^T {\bf A}_{(n)_k} \}_{n \in \N}$. This property extends to all ${\bf y} \in K$. Since $W$ spans $V \supset \widehat{V}$, it follows that one of these scalar products is in fact strictly positive unless ${\bf y} = 0$. From this observation we conclude that $K \cap (-K) = \{0\}$.
By construction, the set $U$ contains a basis of of $\widehat{V}$. These basis vectors span a cone that lies in $K$. Hence, $K$ has nonempty interior within $\widehat{V}$ and it follows that $K$ is indeed a cone.
\end{proof}

In what follows, we let $\rho^{\ast} = \rho^{\ast}(\widehat{\mathcal{A}})$ be the joint spectral radius of the set $\widehat{\mathcal{A}}$ of matrices $\widehat{\bf{A}}_i$ with $0\leqslant i \leqslant k-1$.
Note that the matrices in $\widehat{\mc A}$ sum up to
$\widehat{\bf A} = {\bf P}_{\widehat{V}} \widetilde{\bf A} {\bf P}_{\widehat{V}}$. By construction of ${\bf Pv}$, the restriction of $\widetilde{\bf A}$ to $\widehat{V}$ is easily seen to inherit the full spectral radius. That is, we have $\rho(\widetilde{\bf A}) = \rho(\widehat{\bf A})$ and keep writing $\rho$ for their common value.
Recall that $\rho^{\ast} \leqslant \rho$ by the fundamental inequality. 
The condition that this inequality is strict gives a spectral characterisation for the ghost measures.

\begin{proof}[Proof of Theorem \ref{PROP:continuity-condition}]
First, let $\rho^* < \rho$, $y \in \TT$, $h \in G$ and recall the expression for $\mu_h(\{y\}) = \lim_{n \to \infty} c_n(h,y)$ in Corollary~\ref{COR:point-mass-sequence}. We get from \eqref{EQ:tilde-replacement} and \eqref{EQ:hat-replacement} that
\[
{\bf u}^T {\bf A}_{x_1} \cdots {\bf A}_{x_{n}} {\bf R}^{n} {\bf R}_h {\bf P} {\bf v}
= {\bf u}^T {\bf A}_{x_1} \widehat{\bf A}_{x_2} \cdots \widehat{\bf A}_{x_n} {\bf R}^n {\bf R}_h {\bf Pv}. 
\]
As the vectors in $V_{\lim}$ are uniformly bounded, one has $$\mu_h(\{y\}) = \lim_{n \to \infty} c_n(h,y) = 0.$$ Since both $y$ and $h$ were chosen arbitrarily, all ghost measures are continuous.

Now, assume $\rho^*=\rho$. We show that there exists a ghost measure $\mu$ and $y\in\mathbb{T}$ for which $\mu(\left\{y\right\})>0$. 
Note that \cite[Prop.~1.6]{J2009} implies that for some $d_1>0$ $$\max_{0\leqslant i_1,i_2,\ldots,i_{N}\leqslant k-1}\big\| \widehat{{\bf A}}_{i_1}\widehat{{\bf A}}_{i_2}\cdots\widehat{{\bf A}}_{i_{N}}\big\|\geqslant d_1(\rho^*)^N,$$ for each positive integer $N$. That is, there is an infinite sequence of words $w_N$ (each $w_N$ having length $N$) such that $\big\| \widehat{{\bf A}}_{w_N}\big\|\geqslant d_1(\rho^*)^N$. This implies that some entry of $\widehat{{\bf A}}_{w_N}$ (with respect to an arbitrary basis) grows in the same manner on some subsequence. We use an appropriate choice for such basis vectors below.
Since $\widehat{V}$ is finite-dimensional, it admits a finite set of basis vectors of the form $\widetilde{\bf A}_w {\bf R}_h {\bf P} {\bf v}$. Similarly, by Lemma~\ref{LEM:u-cyclic-span}, we also get a finite set of basis vectors of the form ${\bf u}^T {\bf A}_{w'}$ and since we work on $\widehat{V} \subset V$, we can assume that $w'$ does not start with $0$.
By the pigeonhole principle, there are fixed choices of the words $w,w'$ and of $h \in G$ such that
\begin{equation}
\label{EQ:product-lower-bound}
{\bf u}^T {\bf A}_{w'} \widehat{\bf A}_{w_N} \widetilde{\bf A}_w {\bf R}_h {\bf P} {\bf v} \geqslant d_2 (\rho^{\ast})^N,
\end{equation}
for some fixed $d_2 > 0$ and infinitely many $N$. Up to restricting to a subsequence, we may in fact assume that this holds for all $N$.  
For each $N$, let $m_N$ be such that $(m_N)_k = w' w_N w$ and $\ell_N +1 = |w' w_N w|$ the length of the corresponding string. 
Using \eqref{EQ:tilde-replacement} and \eqref{EQ:hat-replacement}, we see that we can drop all decorations on the left hand side of \eqref{EQ:product-lower-bound}, yielding
\begin{equation}
\label{EQ:product-lower-bound-II}
{\bf u}^T {\bf A}_{(m^{}_N)^{}_k} {\bf R}_h {\bf Pv} \geqslant d_2 (\rho^*)^N.
\end{equation}
Since $m_N$ grows with $N$, the diameter of the intervals $I_N = I_{\ell_N,m_N}$ shrinks to $0$ as $N \to \infty$. 
Recall from the proof of Theorem~\ref{THM:ghost-measure-group} that
\[
\mu_n(I_{\ell,m}) 
\sim \rho^{-\ell} \frac{{\bf u}^T {\bf A}_{(m)_k} {\bf R}^{\ell-n} {\bf Pv}}{{\bf u}^T ({\bf A} - {\bf A}_0) {\bf R}^{-n} {\bf Pv}}
\]
as $n \to \infty$.
Since $\|{\bf R}^{-n}\|=1$ for every $n$, the denominator is uniformly bounded above. 
Hence, there is a constant $c>0$ such that for large enough $n$,
\[
\mu_n(I_{\ell,m}) \geqslant c \rho^{-\ell} {\bf u}^T {\bf A}_{(m)_k} {\bf R}^{\ell-n} {\bf Pv}.
\] 
 Given $N\in\N$, we may therefore choose a large enough $n_N$ with the property that ${\bf R}^{\ell_N - n_N}$ is sufficiently close to ${\bf R}_h$ and such that
\[
\mu_{n_N}(I_N) 
\geqslant \frac{c}{2}  \rho^{-\ell_N} {\bf u}^T {\bf A}_{(m_N)_k} {\bf R}_h {\bf P} {\bf v}
\geqslant \frac{c}{2} d_2 (\rho^*)^{N-\ell_N},
\]
where the last step follows by \eqref{EQ:product-lower-bound-II}. Since $N - \ell_N = |w'| + |w| +1$ is a constant, this yields that there is some constant $d>0$ with $\mu_{n_N} (I_N) > d$ for all $N \in \N$. 
Again restricting to a subsequence if necessary, we assume that $\mu_{n_N}$ converges to some ghost measure $\mu$ and that the (middle points of) $I_N$ converge to some point $y \in \TT$ as $N \to \infty$. Then, for every compact neighbourhood $U$ of $y$, the Portemanteau theorem yields that
\[
\mu(U) \geqslant \limsup_{N \to \infty} \mu_{n_N} (U)
\geqslant \limsup_{N \to \infty} \mu_{n_N}(I_N) > d,
\]
so $\mu(\{y\}) > d$, implying that there is a ghost measure that is not continuous.
\end{proof}

\begin{remark}
Theorem~\ref{PROP:continuity-condition} may be rephrased by saying that $\rho = \rho^*$ whenever there is at least one ghost measure that admits a Bragg peak.
It is natural to ask whether this also implies that there is a ghost measure without a continuous component. However, this is not always the case; see Example~\ref{EX:pp+ac} for a counterexample.  \exend
\end{remark}

In what follows, we assume that $\mu_h$ is continuous and provide criteria for absolute continuity and singularity. 
To do this, we look at the almost everywhere behaviour of regular sequences via the related matrix norms, and so we turn our attention to the Lyapunov exponent, $\log_k\bar{\rho}$.

In our next result, we use the fact that if for $\lambda$-almost all $x$, $\{E_i(x)\}_{i\geqslant 1}$ is a sequence of sets that \emph{shrinks nicely} to $x$ and $\mu(E_i(x))/\lambda(E_i(x))$ converges to zero as $i\to\infty$, then the measure $\mu$ is singular; see Rudin \cite[Theorem 7.14]{Rudin}. Here, the sequence $\left\{E_i(x)\right\}$ shrinks nicely to $x$ if there is an $\alpha>0$ for which the following conditions hold: 
(i) $E_i\subset B(x,r_i)$ where $B(x,r_i)$ the open ball centred at $x$ with radius $r_i$ and (ii) $\lambda(E_i)\geqslant \alpha\cdot \lambda( B(x,r_i))$. 

In what follows, we let $\bar{\rho}=\bar{\rho}(\widehat{\mathcal{A}})$, $\rho^*=\rho^*(\widehat{\mathcal{A}})$ and $\rho=\rho(\widehat{\bf A})$.  We have the following generalisation of Theorem 2.2 in \cite{CEGM}.

\begin{proposition}\label{prop:sc} If   $\bar{\rho}<\rho/k$, then all continuous ghost measures $\mu_h$ are singular.
\end{proposition}

\begin{proof} Let $X\subset[0,1)\setminus \frac{1}{(k-1)}\mathbb{Z}[k^{-1}]$ be a set with $\lambda(X)=1$ for which 

$$\bar{\rho}=\lim_{n\to\infty}\big\|\widehat{{\bf A}}_{w_2}\widehat{{\bf A}}_{w_2}\cdots\widehat{{\bf A}}_{w_{n+1}}\big\|^{1/n},$$
for every $x=\pi(w)\in X$ with $w=w_1w_2w_3\cdots$. Also, let $\varepsilon>0$ be such that $\bar{\rho}+\varepsilon<\rho/k$. Then $x\in I_{\ell,w_1w_2\cdots w_{\ell+1}}$ for each $\ell$, and taking the limit, we have $\{x\}= \bigcap_{\ell}I_{\ell,w_1w_2\cdots  w_{\ell+1}}$. Note that since $x\not\in \frac{1}{(k-1)}\mathbb{Z}[k^{-1}]$, it is never at the border of any $I_{\ell,m_{\ell}}$ for any $\ell$. Taking $r_\ell=k^{-\ell}(k-1)^{-1}$ and $\alpha=\frac{1}{2}$ in the definition above,  one confirms that the sequence $\left\{E_\ell\right\}$ with $E_{\ell}:=I_{\ell,w_1\cdots w_{\ell+1}}$ shrinks nicely to $x$.

 From Corollary~\ref{COR:mu-continuous-on-I}, we have
\[
\mu_h(I_{\ell,w_1w_2\cdots w_{\ell+1}})
=\rho^{-\ell} \frac{{\bf u}^T {\bf A}_{w_1}\widehat{{\bf A}}_{w_2}\cdots\widehat{{\bf A}}_{w_{\ell+1}} {\bf R}^{\ell} {\bf R}_h {\bf Pv}}{{\bf u}^T ({{\bf A}} - {{\bf A}}_0) {\bf R}_h {\bf Pv}}, 
\]
where ${\bf u}^T ({{\bf A}} - {{\bf A}}_0) {\bf R}_h {\bf Pv}\geqslant C>0$, for some $C$ because $f$ is nondegenerate. 
Note that we have $\lambda(I_{\ell,w_1w_2\cdots w_{\ell+1}})=k^{-\ell}(k-1)^{-1}$. Let $N=\max_{1\leqslant i\leqslant k} \|{\bf A}_i\|$  and set 
$d=C^{-1}(k-1)N\|{\bf u}\| \|{\bf Pv}\|$. One obtains 
\begin{align*}
\lim_{\ell\to\infty} \frac{\mu_h(I_{\ell,w_1w_2\cdots w_{\ell+1}})}{\lambda(I_{\ell,w_1w_2\cdots w_{\ell+1}})}&\leqslant d\lim_{\ell\to\infty}\left(\frac{k}{\rho}\right)^{\ell}  \big\|\widehat{{\bf A}}_{w_2\cdots w_{\ell+1}}\big\|\\
& \leqslant d\lim_{\ell\to\infty}\left(\frac{k}{\rho}\cdot (\bar{\rho}+\varepsilon)\right)^{\ell}=0,
\end{align*}
where we have chosen $\|\cdot\|$ to be the $2$-norm so that $\|{\bf R}^{\ell}{\bf R}_h\|=1$, for all $\ell\geqslant 1$ and $h\in G$. This means that $\mu_h$ is singular, from which the claim follows. 
\end{proof}

Each of the values in the fundamental inequality \eqref{eq:fund} reflect a certain type of behaviour of the underlying regular sequence. The Lyapunov exponent $\bar{\rho}$ describes the almost-everywhere behaviour, $\rho/k$ describes the average behaviour, the joint spectral radius $\rho^*$ describes the maximal behaviour, and $\rho$ describes the behaviour of the partial sums. The previous two results considered what happens when either of the two outside inequalities in \eqref{eq:fund} is strict. In the rest of this section, we focus on the middle inequality of \eqref{eq:fund} connecting $\rho/k$ and $\rho^*$. Here, we note that if $\rho/k<\rho^*$, then the maximal values of the associated regular sequence occur on a thin set. This is obtained via the following reformulation of Proposition~6.1 in \cite{CEGM}.

\begin{proposition}\label{prop:af}
Let $f$ be a nonnegative, minimal and nondegenerate $k$-regular sequence. Suppose $\rho/k<\rho^{\ast}$. For $m\geqslant 0$, define $a_m$ to be 
\[
a_m(x)=\frac{1}{\big(\rho^{\ast}\big)^m} f(k^{m}+\lfloor k^{m}(k-1)x\rfloor)
\]
Then, the sequence
$\left\{a_{m}(x)\right\}$ converges to zero as $m\to \infty$ for $\lambda$-a.e.~$x\in[0,1]$.
\end{proposition}

\begin{proof}
Since $\rho$ is the spectral radius of $\widetilde{\bf A}$, we know that there exists a polynomial $P(m)$ such that $\Sigma_f(m)\leqslant P(m)\rho^{m}$, for some polynomial $P(m)$ of degree at most the degree of $f$. We then show that there exist exponentially few integers $j\in [k^{m},k^{m}+1)$ such that $f(j)\geqslant \varepsilon (\rho^*)^{m}$ for any given $\varepsilon>0$. Note that the number $N$ of such integers satisfies $N\varepsilon (\rho^*)^{m}\leqslant P(m)\rho^{m}$. Pick a $c\in(0,1)$ with $\rho/k<c\rho^*$. Then one has 
\begin{equation}
N\leqslant \frac{P(m)\rho^{m}}{(\rho^{\ast})^{m}\varepsilon}=\frac{P(m)}{\varepsilon}\left(\frac{\rho}{\rho^*}\right)^m< \frac{P(m)(ck)^m}{\varepsilon}.
\end{equation}
Thus, for any $d\in(c,1)$ and for large $m$, we have $\lambda(\left\{ x\in [0,1]\colon a_{m}(x)\geqslant \varepsilon\right\})\leqslant d^m/\varepsilon$. 
Since this expression is summable over $m$, it follows by a standard application of the Borel--Cantelli lemma, that $\lim_{m \to \infty}a_m(x) = 0$ for $\lambda$-almost every $x$.
\end{proof}

The following corollary can be proved along the same lines as Proposition \ref{prop:af}.

\begin{corollary}\label{COR:Lyapunov}
Let $f$ be a nonnegative, minimal and nondegenerate $k$-regular sequence. Suppose $\rho/k<\rho^{\ast}$. Then, for Lebesgue a.e. $x\in [0,1]$ with $x=\pi(w)$,
\[
\lim_{n\to \infty} (\rho^{\ast})^{-n}\big\|\widehat{{\bf A}}_{w_1}\widehat{{\bf A}}_{w_2}\cdots \widehat{{\bf A}}_{w_{n}}\big\|=0. 
\]
\end{corollary}

In an ideal situation, one would hope to relate the maximal behaviour of the sequence with the almost-everywhere behaviour. For example, if the maximal behaviour is somehow greater than the average behaviour, as in $\rho/k<\rho^*$, then, as $\rho/k$ is repelling $\rho^*$, we could hope also that $\rho/k$ is repelling $\bar{\rho}$. If this were the case, then we could replace the condition $\bar{\rho}<\rho/k$ in Proposition \ref{prop:sc} with the condition $\rho/k<\rho^*$ and conclude singular continuity. On the other side, if $\rho/k=\rho^*$, then the maximal behaviour is equal to the average behaviour, and we should expect a certain level of smoothness---this is in fact, exactly the case, we get absolutely continuous ghost measures! Of course, this relies on having an ideal situation. It turns out, the situation is ideal enough if the matrices in $\widehat{\mathcal{A}}$ are nonnegative and do not share a nontrivial invariant subspace. 

The following result is a generalisation of Proposition 7.2 of Jungers \cite{J2009}.

\begin{proposition}\label{prop:sc2} Suppose that the matrices in $\widehat{\mathcal{A}}$ are nonnegative and do not share a nontrivial invariant subspace. If $\rho/k<\rho^*$, then all continuous ghost measures $\mu_h$ are singular. 
\end{proposition}

\begin{proof} Using standard results of Perron--Frobenius theory, we let ${\bf x}$ be a (positive) $\rho$-eigenvector of $\widehat{\bf A}$, and we consider the norm $\|\cdot\|=\langle \cdot, {\bf x}\rangle$ on $\mathbb{R}_+^{d}$. Now, pick some $n\geqslant 1$ and ${\bf y}\in\mathbb{R}_+^{d}$ with $\|{\bf y}\|=1$ and such that $$r_n=\left(\prod_{0\leqslant i_1,\ldots,i_n\leqslant k-1}\big\|\widehat{\bf A}_{i_1}\widehat{\bf A}_{i_2}\cdots\widehat{\bf A}_{i_n}{\bf y}\big\|\right)^{1/k^n}$$ is maximal. Then, using the arithmetic-geometric mean inequality, we have 
\begin{multline*}
(\bar{\rho})^n\leqslant  r_n\leqslant \frac{1}{k^n}\sum_{0\leqslant i_1,\ldots,i_n\leqslant k-1}\big\|\widehat{\bf A}_{i_1}\widehat{\bf A}_{i_2}\cdots\widehat{\bf A}_{i_n}{\bf y}\big\|\\
=\frac{1}{k^n}\sum_{0\leqslant i_1,\ldots,i_n\leqslant k-1}\big\langle\widehat{\bf A}_{i_1}\widehat{\bf A}_{i_2}\cdots\widehat{\bf A}_{i_n}{\bf y},{\bf x}\big\rangle=\big\langle{\bf y},k^{-n}\widehat{\bf A}^n{\bf x}\big\rangle\\
=(\rho/k)^n\langle{\bf y},{\bf x}\rangle=(\rho/k)^n.
\end{multline*} But the geometric mean is equal to the arithmetic mean only when all of the $k^n$ terms $\big\|\widehat{\bf A}_{i_1}\widehat{\bf A}_{i_2}\cdots\widehat{\bf A}_{i_n}{\bf y}\big\|$ are equal. Since $\rho/k<\rho^*$, this is not the case for all $n$, and so the second inequality above is strict, and we have that $\bar{\rho}<\rho/k$. An application of Proposition \ref{prop:sc} provides the desired result.
\end{proof}

Recall that a positive measure $\mu$ on a $\sigma$-algebra $\mathfrak{M}$ is {\em absolutely continuous} with respect to $\lambda$ provided $\mu(E)=0$ for every $E\in\mathfrak{M}$ for which $\lambda(E)=0$.  

\begin{proposition}\label{PROP:ac-condition} Suppose that the matrices in $\widehat{\mathcal{A}}$ do not share a nontrivial invariant subspace. If $\rho/k=\rho^*$, then all ghost measures $\mu_h$ are absolutely continuous.  
\end{proposition}

\begin{proof} 
Since the matrices in $\widehat{\mathcal{A}}$ do not share a nontrivial invariant subspace, there is a $c>0$ such that  $$\max_{1\leqslant i_1,i_2,\ldots,i_{n}\leqslant \ell}\big\| \widehat{{\bf A}}_{i_1}\widehat{{\bf A}}_{i_2}\cdots\widehat{{\bf A}}_{i_{n}}\big\|\leqslant c(\rho^*)^n,$$ for each $n\geqslant 1$; see \cite[p.~24]{J2009}. 
Consider the set $\mathcal{I}_k$ which generates the Borel $\sigma$-algebra on $[0,1)$. For $I_{\ell,m}\in\mathcal{I}_k$, with $(m)_k = w_1 \cdots w_{\ell+1}$
\begin{align*}
\mu_h(I_{\ell,m})&= \rho^{-\ell}\frac{{\bf u}^T {\bf A}_{w_1}\widehat{{\bf A}}_{w_2}\cdots\widehat{{\bf A}}_{w_{\ell+1}} {\bf R}^{\ell} {\bf R}_h {\bf Pv}}{{\bf u}^T ({{\bf A}} - {{\bf A}}_0) {\bf R}_h {\bf Pv}}\\
&\leqslant \frac{cN\|{\bf u}\|\|{\bf Pv}\|}{C}\left(\frac{\rho^*}{\rho}\right)^\ell=\frac{c(k-1)\|{\bf u}\|\|{\bf Pv}\|}{C}\lambda(I_{\ell,m}),
\end{align*}
where $N$ and $C$ are the same positive constants as in the proof of Proposition~\ref{prop:sc}, and $\lambda(I_{\ell,m})=1/k^\ell(k-1)$. Thus $\mu_h(A)\leqslant C^{\prime}\lambda(A)$,
for some positive constant $C^{\prime}$ and for each $A$ in the algebra generated by $\mathcal{I}_k$ via finite unions. By the monotone class theorem, this inequality is valid for the whole $\sigma$-algebra, so that $\mu_h$ is absolutely continuous with respect to $\lambda$.
\end{proof}

Note the absence of the nonnegativity condition in Proposition \ref{PROP:ac-condition}. Propositions \ref{prop:sc2} and \ref{PROP:ac-condition} combine to prove Theorem \ref{thm:specclass}.

\section{Structure of ghost measures}\label{sec:structure}

As before, let $f$ be a nonnegative and nondegenerate $k$-regular sequence with linear representation $({\bf u}, \mathcal{A}, {\bf v})$.
In this section, we aim to provide an infinite convolution representation for ghost measures $\mu_h$ under the additional assumption that ${\bf A}_i$ commutes with ${\bf R}$ for all $i$ (which holds in particular if $\rho$ is the unique eigenvalue of maximal modulus for $\widetilde{\bf A}$). 
Recall from Corollary~\ref{COR:reduced-form} that in this case there is a reduced representation $({\bf w},\mathcal{B},{\bf x})$ of a sequence $g = g_h$ such that $\mu_h$ is the unique ghost measure of $g$. 
We therefore restrict our attention to reduced representations of a sequence $g$ in this and the following section and denote its unique ghost measure by $\mu_g$.

As in Section~\ref{sec:spectype}, it will prove useful to consider the restrictions $\widehat{\bf B}_i$ of $\widetilde{\bf B}_i$ to the space $\widehat{V}$, spanned by $$U = \{ \widetilde{ \bf B}_0^n \widetilde{ \bf B}_{(m)_k}{\bf x} \}_{m,n \in \N_0} .$$
Recall that these matrices admit some invariant cone $K$. The cone condition acts as a replacement for the condition of nonnegativity of the digit matrices in Coons, Evans and Ma\~nibo \cite{CEM1}. 

\begin{lemma}\label{lem:B} 
Let $({\bf w},\mathcal{B},{\bf x})$ be the reduced representation of a sequence $g$. Then, the sequence $(\rho^{-1}\widehat{\bf B})^N$ converges to a projector onto the $\rho$-eigenspace of $\widehat{\bf B}$. In particular, there is a $B>0$, such that $\big\|(\rho^{-1}\widehat{\bf B})^N\big\|<B,$ for all integers $N\geqslant 0$. 
\end{lemma}

\begin{proof}
Recall that $\rho = \rho(\widetilde{\bf B}) = \rho(\widehat{\bf B})$ is the spectral radius of $\widehat{\bf B}$.
In view of Lemma~\ref{LEM:RA-limit}, it suffices to show that if $\lambda$ is an eigenvalue of $\widehat{\bf B}$ with $|\lambda| = \rho$, every corresponding Jordan block must be trivial. 
Otherwise, $\| \widehat{\bf B}^n \| \geqslant c n^r \rho^n$ for some $c>0$ and $r \in \N$. In this case, we can choose basis vectors ${\bf w}^T {\bf B}_{w'}$ of $V \supset \widehat{V}$ and $\widetilde{\bf B}_w {\bf x} \in \widehat{V} $ such that 
\[
|{\bf w}^T {\bf B}_{w'} \widehat{\bf B}^n \widetilde{\bf B}_w {\bf x} | \geqslant c' n^r \rho^n,
\]
for some $c' > 0$ and infinitely many $n$. But the left hand side is equal to
\[
{\bf w}^T {\bf B}_{w'} {\bf B}^n {\bf B}_w {\bf x}
\leqslant {\bf w}^T {\bf B}^{n +|w'w|} {\bf x} = \rho^{n + |w'w|} {\bf w}^T {\bf x},
\]
a contradiction. Hence, $r=0$ and the claim follows.
\end{proof}

For the matrices $\widehat{\bf B}_0,\ldots,\widehat{\bf B}_{k-1}$ coming from the reduced representation of $g$, set 
$$\widehat{\bf B}(z):=\frac{1}{\rho}\sum_{j=0}^{k-1}\widehat{\bf B}_j z^j,$$ 
and note that $\rho \widehat{\bf B}(1)=\sum_{j=0}^{k-1} \widehat{\bf B}_j = \widehat{\bf B}$. 
For a Dirac measure $\delta_x$ on $x \in \TT$ and $j \in \N$ we denote by $\delta_x^j = \delta_x \ast \ldots \ast \delta_x = \delta_{jx}$ the $j$-fold convolution product of $\delta_x$ with itself.
With this notation, we can give an alternative expression for the approximants of the ghost measure.

\begin{lemma}
Let $({\bf w}, \mathcal{B}, {\bf x})$ be a reduced representation of $g$. Then, the $n$-th approximant is given by
\[
\mu_n = {\bf w}^T ({\bf B} - {\bf B}_0)\big(\delta_{1/(k-1)}\big) *\bigg(\bigast_{r=1}^{n} \widehat{\bf B} \big(\delta_{1/k^r(k-1)}\big)\bigg)\, {\bf x},
\]
 where $({\bf B}-{\bf B}_0)(z):=\frac{1}{\rho}\sum_{j=1}^{k-1}{\bf B}_j  z^{j-1}.$
\end{lemma}

\begin{proof}
Since the representation is reduced, we have $\Sigma_g(n) = \rho^n$ and hence 
\[
\mu_n =  \sum_{m = k^n}^{k^{n+1} - 1} \rho^{-n} g(m) \delta_{q(m)}
\]
where $q(m) = (m-k^n)/(k^n(k-1))$.
The convolution product is a finite sum of Dirac measures, positioned at
\[
p(i_0 \cdots i_n) = \frac{i_0 - 1}{k-1} + \frac{i_1}{k(k-1)} + \ldots + \frac{i_n}{k^n(k-1)},
\]
 with weights
\[
\rho^{-n} {\bf w}^T {\bf B}_{i_0} \widehat{\bf B}_{i_1} \cdots \widehat{\bf B}_{i_n} {\bf x} = \rho^{-n} g(m),
\]
where $k^n \leqslant m < k^{n+1}$ is the unique integer with $(m)_k = i_0 \cdots i_n$. With this notation, may rewrite $q(m)$ as
\[
q(m) = \frac{1}{k^n(k-1)} \sum_{r=0}^{n} (i_r - \delta_{r,0}) k^{n-r} = p((m)_k),
\] 
which shows that $\mu_n$ and the convolution product consist of Dirac measures at the same positions with identical weights.
\end{proof}

We already know that $\mu_n$ converges weakly to the unique ghost measure of $g$. However, a priori, it is not clear that the same holds for the corresponding matrix convolution product without the projection via ${\bf w}$ and ${\bf x}$. The remainder of this section is devoted to showing that this infinite convolution product indeed exists. 

\begin{proposition}[Structure]\label{prop:mainKmu} Let $({\bf w},\mathcal{B},{\bf x})$ be the reduced representation of a sequence $g$. Then, the unique ghost measure of $g$ satisfies
 \begin{equation}\label{eq:convprod} 
\mu_g={\bf w}^T({\bf B}-{\bf B}_0) \big(\delta_{1/(k-1)}\big)*\bigg(\bigast_{n=1}^{\infty} \widehat{\bf B} \big(\delta_{1/k^n(k-1)}\big)\bigg)\, {\bf x}.
 \end{equation} 
\end{proposition} 

Proposition~\ref{prop:mainKmu} comes as a consequence of L\'evy's continuity theorem, which states that a uniformly bounded sequence of measures $\mu_N$ on $\mathbb{T}$ weakly converges to $\mu$ if and only if for each $m\in\mathbb{Z}$, $\lim_{N\to\infty}\FT[{\mu_N}](m)$ exists and equals $\FT[\mu](m)$, where $\FT$ denotes the standard Fourier transform.

We require the following lemma concerning the Fourier transform.

\begin{lemma} For any positive integers $M< N$, the matrix products $$\Pi_{M}^{N}(t):=\prod_{n=M}^{N-1} \widehat{\bf B}\big(e^{-\frac{2\pi i t}{k^n(k-1)}}\big),$$ as a function of $t$, are uniformly bounded.
\end{lemma}

\begin{proof} Let $S\subset\mathbb{T}$ and $\|\cdot\|$ be the standard Frobenius norm, that is, the square root of the sum of the squares of the entries. Using Lemma~\ref{LEM:cone-condition} and Corollary~\ref{cor:Bcone}, we have 
\begin{align*}\Big\|\Pi_{M}^{N}(t)\Big\|
&=\left\| \prod_{n=M}^{N-1} \widehat{\bf B}\big(e^{-\frac{2\pi i t}{k^n(k-1)}}\big)\right\| 
\leqslant \sum_{i_1,\ldots,i_{N-M}}\left\|\prod_{j=1}^{N-M}\frac{1}{\rho} \widehat{\bf B}_{i_j}\right\|\\
& \leqslant \frac{1}{c}\left\|\sum_{i_1,\ldots,i_{N-M}}\prod_{j=1}^{N-M}\frac{1}{\rho} \widehat{\bf B}_{i_j}\right\|=\frac{1}{c}\left\|\left(\frac{1}{\rho}\widehat{\bf B}\right)^{N-M}\right\|\leqslant \frac{B}{c},
\end{align*} 
where the sums range over all possible values $i_1,\ldots,i_{N-M}\in\{0,1,\ldots,k-1\}.$
\end{proof}

\begin{proof}[Proof of Proposition \ref{prop:mainKmu}] 
Let $m$ be an integer. Appealing to L\'evy's continuity theorem, to prove the result, it is enough to prove that the infinite product $\Pi_1^\infty(m)$ converges for each integer $m$.

To this end, write $\widehat{\bf B}_n(m):= \widehat{\bf B}\big(e^{-\frac{2\pi i m}{k^n(k-1)}}\big),$ and note that there is some constant $C=C(m)>0$ depending on $m$ such that $$\left\|\widehat{\bf B}_n(m)-\frac{1}{\rho}\widehat{\bf B}\right\|\leqslant \frac{C}{k^n}.$$ For convenience, we will suppress the dependence on $m$ in $\Pi_M^N = \Pi_M^N(m)$. Using that $\Pi_{M}^{P}=\Pi_{M}^{N}\Pi_{N}^{P}$ for $M\leqslant N\leqslant P$, along with a telescoping argument, we have
\begin{multline*}\left\|\Pi_{M}^{N}-\left(\frac{1}{\rho} \widehat{\bf B}\right)^{N-M}\right\| 
=\left\|\sum_{i=0}^{N-M-1}\left(\frac{1}{\rho}\widehat{\bf B}\right)^i\left(\widehat{\bf B}_{M+i}(m)-\frac{1}{\rho} \widehat{\bf B}\right)\Pi_{M+i+1}^{N}\right\|\\
\leqslant\sum_{i=0}^{N-M-1}\left\|\left(\frac{1}{\rho}\widehat{\bf B}\right)^i\right\|\left\|\widehat{\bf B}_{M+i}(m)-\frac{1}{\rho}\widehat{\bf B}\right\|\Big\|\Pi_{M+i+1}^{N}\Big\|\\
\leqslant \frac{B^2}{c}\sum_{i=0}^{N-M-1}\left\| \widehat{\bf B}_{M+i}(m)-\frac{1}{\rho}\widehat{\bf B}\right\|\leqslant \frac{B^2C}{c}\sum_{i=0}^{N-M-1}\frac{1}{k^{M+i}}<\frac{2B^2C}{k^M}.
\end{multline*} Combining this inequality with the fact that as $N\to\infty$, $(\rho^{-1}\widehat{\bf B})^{N-M}$ converges, we obtain $$\lim_{N,N'\to\infty}\Big\|\Pi_{M}^{N}-\Pi_{M}^{N'}\Big\|\leqslant \frac{E}{k^M}$$ for some constant $E=E(m)>0$ depending only on $m$. 

Finally, let $\varepsilon>0$ be given and $N$ be large enough so that $\varepsilon>EB/ck^N$ and $N\geqslant M\geqslant 1$. Then $$\lim_{P,P'\to\infty}\Big\|\Pi_{M}^{P}-\Pi_{M}^{P'}\Big\|=\lim_{P,P'\to\infty}\Big\|\Pi_{M}^{N}(\Pi_{N}^{P}-\Pi_{N}^{P'})\Big\|\leqslant \frac{B}{c}\cdot\frac{E}{k^N}<\varepsilon,$$ which implies the convergence of the infinite product.
\end{proof}

\section{Spectral purity}\label{sec:ergspecpur}

 Without any further restrictions, we cannot expect to have spectral purity for the unique ghost measure $\mu_g$ in the previous section. This is because it might be possible to decompose a (reduced) representation into several blocks that give rise to different spectral components.
To avoid this effect, we assume that the spectral radius $\rho$ is a \emph{simple} eigenvalue of $\widehat{\bf B}$.
Under this condition, we show that the unique ghost measure  $\mu_g$ is spectrally pure, that is, $\mu_g$ is either pure point, absolutely continuous, or singular continuous. 
For this, we leverage the structure of $\mu_g$ in Proposition~\ref{prop:mainKmu} and the one-dimensionality of the $\rho$-eigenspace of $\widehat{\bf B}$.

\begin{proof}[Proof of Theorem \ref{thm:specpure}]
First, given a vector-valued measure ${\boldsymbol \nu}$, let us define 
\[
K \colon {\boldsymbol \nu} \mapsto \widehat{\bf B}(\delta_{1/k(k-1)}) \ast \bigl( {\boldsymbol \nu} \circ S_k^{-1} \bigr),
\]
where $S_k \colon [0,1) \to [0,1), x \mapsto x/k$. Iterating this operator, we obtain,
\[
K^n({\boldsymbol \nu}) = {\boldsymbol \mu}_n \ast \bigl({\boldsymbol \nu} \circ S_k^{-n} \bigr),
\]
where
\[
{\boldsymbol \mu}_n = \bigast_{m=1}^n \widehat{\bf B}(\delta_{1/k^m(k-1)}).
\]
By the previous section, this convolution converges weakly and the limit ${\boldsymbol \mu}$ satisfies
\[
{\boldsymbol \mu} = \lim_{n \to \infty}\left( \bigast_{m=1}^n \widehat{\bf B}(\delta_{1/k^m(k-1)})
\ast
\bigast_{m= n+ 1}^{2n} \widehat{\bf B}(\delta_{1/k^m(k-1)})\right)
= {\boldsymbol \mu} {\bf P};
\]
compare the proof of Proposition~\ref{prop:mainKmu}.
At the same time, a direct calculation yields that
\[
\lim_{n \to \infty} {\boldsymbol \nu} \circ S_k^{-n}
= \delta_0 {\boldsymbol \nu}(\TT),
\]
where ${\boldsymbol \nu}(\TT)$ denotes the total mass vector.
Hence, 
\[
\lim_{n \to \infty} K^n({\boldsymbol \nu})
= {\boldsymbol \mu} {\bf P} {\boldsymbol \nu}(\TT)
\sim  {\boldsymbol \mu} {\bf x},
\]
since ${\bf x}$ is the unique $\rho$-eigenvector of $\widehat{\bf B}$. 
Hence, if $K({\boldsymbol \nu}) = {\boldsymbol \nu}$, it follows that ${\boldsymbol \nu} \sim {\boldsymbol \mu} {\bf x}$. 
That is, $K$ has a one-dimensional eigenspace of measures, spanned by ${\boldsymbol \mu} {\bf x}$. 
It is straightforward to see, that $K$ maps a vector of pure point measures to a vector of pure point measures (and the same holds for the other spectral types). Hence, if $K({\boldsymbol \nu}) = {\boldsymbol \nu}$, we also get, using the linearity of $K$,
\[
K({\boldsymbol \nu}_{\bullet}) = {\boldsymbol \nu}_{\bullet}
\]
for each $\bullet \in \{{\rm pp,sc,ac} \}$. Since the solution space is one-dimensional, only one of the components can be nonzero.
 The same holds for $\mu_g$ since it can be written as the convolution product of ${\boldsymbol \mu} {\bf x}$ with a finite sum of vector-valued Dirac measures; see Proposition~\ref{prop:mainKmu}.
\end{proof}

\begin{remark}
The proof above is an analogue of the proof of spectral purity for (one-dimensional)
infinite convolution products (which include Bernoulli convolutions) and self-similar measures.  We briefly note that ghost measures are not self-similar in general. 
As an example, consider the ghost measure $\mu_z$ of the $2$-Zaremba sequence, the $2$-regular sequence $f$ admitting a linear representation
\[
{\bf u}={\bf v}^{T}=(1,0), \quad  {\bf B}_0=\begin{pmatrix}
1 & 1\\
1 & 0 
\end{pmatrix} \quad \text{and} \quad 
{\bf B}_1=\begin{pmatrix}
2 & 1\\
1 & 0 
\end{pmatrix}.
\]
Its ghost measure is neither self-similar nor a (scalar) infinite convolution. The graph of its distribution function is an affine image of a section of a three-dimensional self-affine measure; see \cite[Sec.~4]{CEGM}. If the matrices are $1\times 1$, $f$ is called a \emph{Salem sequence} and the resulting measure $\mu_f$ is self-similar; compare \cite[Prop.~4]{CEGM}. This case includes missing-digits measures and their weighted variations. \exend
\end{remark}

\section{Strong finiteness property for pure point ghost measures}\label{sec:rational}

In this section, we consider the special case that the $k$-regular sequence $f$ can be represented via nonnegative matrices. More precisely, let $({\bf u},\mathcal{A},{\bf v})$ be a minimal representation of $f$ and recall that $\widetilde{\bf A}_i$ is the restriction of ${\bf A}_i$ to the subspace
\[
V = {\rm span} \{ {\bf u}^T {\bf A}_{(m)_k} \}_{m \in \N}.
\]
We set $\widetilde{\mc A} = \{ \widetilde{\bf A}_0,\ldots, \widetilde{\bf A}_{k-1} \}$ and recall $\widetilde{\bf A} = \sum_{j=0}^{k-1} \widetilde{\bf A}_j$. 

\begin{definition}
We call a nondegenerate $k$-regular sequence $f$ \emph{positively presentable} if there exists a minimal representation of $f$ such that all matrices in $\widetilde{\mc A}$ and ${\bf v}$ are nonnegative (with respect to an appropriate basis of $V$).
\end{definition}

In the remainder of this section, we let $f$ be a positively presentable $k$-regular sequence with minimal representation $({\bf u},\mathcal{A},{\bf v})$. Also, we assume without loss of generality, that an eventual change of basis on $V$ has been performed and that all matrices in $\widetilde{\mc A}$ are nonnegative with respect to basis vectors ${\bf e}_i$, with $1\leqslant i \leqslant d'$, where $d' \in \{d-1,d\}$.

Some of the already established results on ghost measures take a slightly easier form in this case. We spell this out explicitly for the reader's convenience. Note that this result was previously recorded in the Introduction as Theorem \ref{thm:probnonneg}.

\begin{theorem}
\label{PROP:ghost-measures-in-rational-case}
Let $f$ be positively presentable and $q$ be the period of $\widetilde{\bf A}$. Then, the set of ghost measures is given by $\{ \mu_{f,1},\ldots \mu_{f,q} \}$, where
\[
\lim_{n \to \infty} \mu_{qn + j} = \mu_{f,j},
\]
for all $1\leqslant j \leqslant q$.
\end{theorem}

\begin{proof}
Since $\widetilde{\bf A}$ is assumed to be nonnegative, the corresponding peripheral spectrum $\sigma_p$ is cyclic, that is, $\sigma_p/\rho$ is the group $G_q$ of the $q$-th roots of unity. From this, it is straightforward to see that the rotation group $G$, generated by $g$, is also cyclic and in fact isomorphic to $G_q$. Hence, by Theorem~\ref{THM:ghost-measure-group}, the ghost measures can be parametrised by $G_q$, which has cardinality $q$.
Finally, recall from the proof of Theorem~\ref{THM:ghost-measure-group} that $(\mu_{n_j})_{j \in \N}$ converges to a ghost measure if $(g^{n_j})_{j \in \N}$ converges in $G$. This shows the convergence of the $q$-periodic subsequences $(\mu_{qn + j})_{n \in \N}$.
\end{proof}

Similar to the evolution of Markov processes, the structure of matrix products over the finite set $\widetilde{\mc A}$ can be described via an associated graph. We introduce the appropriate notions in the following and recall some basic notation. 

Let $n \in \N$ and $\C{C} = ( {\bf C}_0,\ldots, {\bf C}_{n-1} )$ be a finite tuple of nonnegative matrices on $\R^{d'}$.  
We assign a directed (multi-)graph $G(\C{C})$ as follows. There are $d'$ vertices, representing the $d'$ elementary unit vectors ${\bf e}_i$, $i=1,\ldots,d'$. There is an edge from $i$ to $j$, labeled by $m$, if and only if $({\bf C}_m)_{ij} > 0$. In this case, we call $({\bf C}_m)_{ij}$ the {\em weight} of the edge $(i,j;m)$. 
A {\em path} from vertex $i$ to vertex $j$ in $G(\C{C})$ is a valid concatenation of edges (every edge ends at the starting point of the next edge), starting from $i$ and ending at $j$. The weight of this path is the product of the weights of the corresponding edges. The label of the path is given by the ordered collection of the edge labels. Hence,
\[
{\bf e}_i^T {\bf C}_{i_1} \cdots {\bf C}_{i_n} {\bf e}_j
\]
is the sum of all weights over paths from $i$ to $j$ that are labeled by $(i_1,\ldots,i_n)$.
We say that a vertex $j$ is {\em reachable} from a vertex $i$ if there is a path from $i$ to $j$. By convention, every vertex is reachable from itself.
Two vertices $i,j$ are {\em strongly connected} if $i$ is reachable from $j$ and vice versa. Note that this defines an equivalence relation. Every corresponding equivalence class is called a \emph{strongly connected component} (SCC). An SCC is called {\em trivial} if it consists of a single vertex $i$ and there is no edge from $i$ to itself. 

There exists a (not uniquely determined) {\em topological ordering} of the SCCs $V_1,\ldots,V_r$ in such a manner that $V_m$ is not reachable from $V_n$ whenever $m < n$. For each SCC $V_i$ there is a corresponding {\em SCC subspace}
\[
W_i = \operatorname{span} \{ {\bf e}_j \}_{j \in V_i}.
\]

If ${\bf C} = \sum_{m=0}^{n-1} {\bf C}_m$, we note that $G(\C{C})$ and $G({\bf C}) : = G(\{\bf C\})$ have a similar structure. We obtain the edges in $G({\bf C})$ from the edges in $G(\C{C})$ by merging all edges that have the same source and target and summing the corresponding weights. In particular, the SCCs are the same for both graphs. Restricted to an SCC subspace $W_i$, the matrix ${\bf C}$ is either irreducible or vanishing. Permuting the Euclidean basis vectors according to the order of the SCCs, we obtain an (upper) block-triangular form
\begin{equation}
\label{EQ:upper-triangular}
{\bf C} = \begin{pmatrix}
{\bf C}_{11} & \cdots & {\bf C}_{1r}\\
\vdots & \ddots &\vdots \\
0 &  \cdots & {\bf C}_{rr}
\end{pmatrix},
\end{equation}
where ${\bf C}_{ii}$ is the restriction of ${\bf C}$ to $W_i$. By standard linear algebra, the collection of eigenvalues of ${\bf C}$ (with multiplicities) coincides with the union over the corresponding collections of eigenvalues of the ${\bf C}_{ii}$. In particular,
\[
\rho({\bf C}) = \max_{1\leqslant i \leqslant r} \rho({\bf C}_{ii}).
\]
We call an SCC $V_i$ \emph{dominant} if $\rho({\bf C}_{ii}) = \rho({\bf C})$. By the Perron--Frobenius theorem, this is a simple eigenvalue of ${\bf C}_{ii}$ and there is a corresponding eigenvector of ${\bf C}_{ii}$ with strictly positive entries on $W_i$. This eigenvector is unique up to a multiplicative constant. If $\rho({\bf C})$ has a nontrivial multiplicity, there are several dominant SCCs. 

Sometimes it will be useful to consider the graph associated to a higher power of $\C{\bf C}$.
We call $n$ a {\em return time} of the vertex $i$ if there is a path of lenth $n$ from $i$ to itself. The greatest common divisor of all return times is the {\em period} of a vertex. Note that the period is the same for all vertices in an SCC. We may therefore also speak of the period of an SCC. If $V_i$ is an SCC of period $p$, then it decays into $p$ subsets $V_{i,1}, \ldots, V_{i,p}$ which are SCCs of $G(\C{\bf C}^{np})$ for every $n \in \N$ (and not reachable from each other in $G(\C{\bf C}^{np})$). 
Taking $n$ large enough, we can even ensure that every pair of vertices in $V_{i,j}$ is connected by an edge in $G(\C{\bf C}^{np})$. This means that the restriction of $\C{\bf C}^{np}$ to $W_i$ decays into a direct sum of strictly positive matrices, each supported on one of the sets $V_{i,j}$; we refer to \cite{Seneta81} for background and a more detailed discussion.
Hence, there exists a power $r$ such that the restriction of $\C{\bf C}^r$ to each of its dominant SCC subspaces is strictly positive. We refer to this property of $\C{\bf C}^r$ as being {\em dominant positive}.

In the following, we use this graph-theoretic approach to construct examples of ghost measures with a Bragg peak at any given rational location. In order to illustrate the procedure, we begin with an example before treating the general case. 

\begin{example}
\label{EX:mu-delta}
Given $k =2$, consider the rational point $y = \pi(x)$ with $x= 1(10)^{\infty}$.
Our goal is to construct a $2$-regular sequence with a ghost measure that has a Bragg peak at $x$. This works in the following manner. We first construct a directed multi-graph $G$ such that we can obtain $x$ by following the labels of the edges and then choose an appropriate collection $\C{C}$ of nonnegative matrices such that $G = G(\C{C})$. Finally, the elements of $\mc C$ will be used to construct a linear representation. A working example for the present case would be
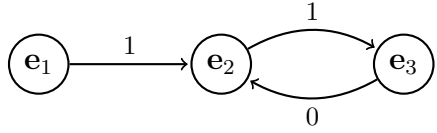
\begin{center}
\begin{tikzpicture}[->, shorten >=1pt, thick, main node/.style={circle,draw,font=\Large\bfseries}, scale = 0.8]

\node[main node](1) at (0,0) {${\bf e}_1$};
\node[main node](2) at (3,0) {${\bf e}_2$};
\node[main node](3) at (6,0) {${\bf e}_3$};

\draw (1) to node[above]{$1$} (2);
\draw (2) to[out = 30, in = 150] node[above]{$1$} (3);
\draw (3) to[out = 210, in = 330] node[below]{$0$} (2);

\end{tikzpicture}
\end{center}
The vertices represent the Euclidean unit vectors. Hence, we will require a $3$-dimensional space for this example. Note that $\{{\bf e}_2, {\bf e}_3\}$ is the only nontrivial strongly connected component. 
In the next step, we choose $\C{C} = ({\bf A}_0,{\bf A}_1)$ such that the graph above coincides with $G(\C{C})$.
This can be done by taking the corresponding incidence matrices
\[
{\bf A}_0 = \begin{pmatrix}
0 & 0 & 0\\
0 & 0 & 0\\
0 & 1 & 0
\end{pmatrix}
\qquad\mbox{and}\qquad
{\bf A}_1 = \begin{pmatrix}
0 & 1 & 0\\
0 & 0 & 1\\
0 & 0 & 0
\end{pmatrix},
\]
Let $f$ be the $2$-regular sequence with linear representation $({\bf e}_1,\{{\bf A}_0,{\bf A}_1\},{\bf v})$, with ${\bf v} = (1,1,1)$. 
Note that for a fixed $n \in \N$
\[
{\bf e}_1^T {\bf A}_{i_1} \cdots {\bf A}_{i_{n}} {\bf v},
\]
corresponds to the number of paths with label $(i_1,\ldots,i_{n})$ that start from ${\bf e}_1$. By construction, this is nonzero if and only if $i_1 \cdots i_n$ is a prefix of $x$. That is, for $k^{n-1} \leqslant m < k^n$, we find $f(m) = 0$, unless $m$ is given by
\[
m_{n} = \sum_{j=1}^n x_j k^{n-j}.
\]
In particular, 
\[
\mu_{n-1} = \delta_{y_n},
\quad y_n = \frac{m_n - k^{n-1}}{(k-1)k^{n-1}} = \sum_{j=1}^n \frac{x_j - \delta_{j,1}}{(k-1)k^{j-1}}. 
\]
By definition of the coding $\pi$, this converges to $\pi(x) = y$ as $n \to \infty$ and thus
\[
\lim_{n \to \infty} \mu_n = \delta_y,
\]
which is hence the unique ghost measure for $f$.
\exend
\end{example}

The ideas above can easily be adapted to any rational number in $[0,1]$. 

\begin{proposition}
Let $y \in [0,1] \cap \Q$ and $k \in \N$ be given. Then, there exists a $k$-regular sequence $f$ such that $f$ has a unique ghost measure, given by $\delta_y$. 
\end{proposition}

\begin{proof}
The ideas are the same as for the example above. Recall that $y \in [0,1]$ is rational if and only if it has a coding $x \in \pi^{-1}(y)$ that is eventually periodic. If there are two possible codings, we pick one arbitrarily. Hence $x$ is of the form
\[
x = u \tilde{u}^{\infty}
\]
for some $u = u_1 \cdots u_{\ell}$ and some suffix $\tilde{u} = u_r \cdots u_{\ell}$ of $u$. We construct a multi-graph $G$ with vertices ${\bf e}_1, \ldots, {\bf e}_{\ell}$ and introduce an edge from ${\bf e}_i$ to ${\bf e}_{i+1}$, labeled by $u_i$ for all $1\leqslant i \leqslant \ell - 1$. Further, we connect ${\bf e}_{\ell}$ to ${\bf e}_r$ with an edge labeled by $u_{\ell}$. This ensures that there is precisely one path of length $n \in \N$ in $G$ starting from ${\bf e}_1$, labeled by the corresponding prefix of $x$. 
For $0 \leqslant i \leqslant k-1$, let ${\bf A}_i$ be the incidence matrix of the edges with label $i$, that is, we set $({\bf A}_i)_{st} = 1$ if there is an edge from $s$ to $t$, labeled by $i$ and $({\bf A}_i)_{st} = 0$ otherwise. 
Now, let $f$ be the $k$-regular sequence with linear representation $({\bf e}_1,\mathcal{A},{\bf v})$, where ${\bf v} = (1,\ldots,1)^T$. Precisely as in Example~\ref{EX:mu-delta}, it follows that the ghost measure approximants are of the form $\mu_n = \delta_{y_n}$ with $\lim_{n \to \infty} y_n = y$ which concludes the proof.
\end{proof}

In the opposite direction, we will show that no ghost measure can have a Bragg peak at an irrational position. To this end, we regard $G(\C{C})$, where
$\mc C$ is given by
\begin{equation}
\label{EQ:C-choice}
\mc C = (\widetilde{\bf A}_0,\ldots,\widetilde{\bf A}_{k-1} ),
\end{equation}
in which case ${\bf C} = \widetilde{\bf A}$. 
We emphasise that the results in the remainder of this section hold for completely arbitrary collections of nonnegative matrices. The notational convention taken in \eqref{EQ:C-choice} is simply for the sake of definiteness and to build a bridge to the application we have in mind. 

To simplify the setting, we first restrict to the case that $\widetilde{\C{\bf A}}$ is dominant positive. In what follows, by $\widetilde{\C{A}}^n$, we mean the set $$\widetilde{\C{A}}^n:=\big\{\widetilde{\bf A}_{i_1} \cdots \widetilde{\bf A}_{i_n} : \widetilde{\bf A}_{i_j} \in \widetilde{\C{A}} \mbox{ for all } 1 \leqslant j \leqslant n \big\}.$$
For the following result, recall the notation for the block triagonal structure of the matrix $\widetilde{\bf A} = {\bf C}$ in \eqref{EQ:upper-triangular}.

\begin{lemma}
\label{LEM:block-dichotomy}
Assume that ${\bf C}_{ii}$ is strictly positive. Then, we have the following dichotomy. Either, there is some $j$ such that $\widetilde{{\bf A}}_j |_{W_i} = {\bf C}_{ii}$, or there is a constant $c\in(0,1)$ such that for every ${\bf D} \in \widetilde{\C{A}}^3$,
\[
{\bf D}|_{W_i} \leqslant c \cdot{\bf C}^3_{ii},
\] 
where this inequality is to be understood element-wise. In particular, in the former case there is an element ${\bf D}' \in \widetilde{\mc A}^3$ such that ${\bf D}'|_{W_i} = {\bf C}_{ii}^3$.
\end{lemma}

\begin{proof}
Assume the first condition does not hold and let ${\bf E} = {\bf C}_{ii}$. Then, for every $j$ there exists a pair $(m_j,n_j)$ such that the corresponding entry of $\widetilde{\bf A}_j$ is strictly smaller than that of ${\bf E}$. For an arbitrary pair $(s,t) \in V_i^2$ we get
\[
({\bf E}^3)_{st} = \sum_{m,n} {\bf E}_{sm} {\bf E}_{mn} {\bf E}_{nt}.
\]
On the other hand, for ${\bf D} = \widetilde{\bf A}_{j_0} \widetilde{\bf A}_j \widetilde{\bf A}_{j_1} \in \C{A}^3$,
\[
{\bf D}_{st} \leqslant \sum_{m,n} {\bf E}_{sm} (\widetilde{\bf A}_j)_{mn} {\bf E}_{nt} < ({\bf E}^3)_{s,t},
\]
since ${\bf E}_{s m_j},{\bf E}_{n_j t} > 0$. Hence for all ${\bf D} \in \widetilde{\C{A}}^3$ and $(s,t) \in V_i$ there exists a constant $0 < c(s,t,{\bf D}) < 1$ such that ${\bf D}_{st} \leqslant c(s,t,{\bf D}) \cdot ({\bf E}^3)_{st}$. Taking the maximum over all triples $(s,t,{\bf D})$ yields the assertion.
\end{proof}

Now we turn to the case of general nonnegative matrices $\widetilde{\bf A}$.
Let $m'$ be minimal with the property that $\widetilde{\bf A}^{m'}$ is dominant positive. 
Let $m = 3 m'$ and observe that the dominant blocks of $\widetilde{\bf A}^m$ are just the cubes of the dominant blocks of $\widetilde{\bf A}^{m'}$. We denote the matrices appearing in the corresponding block decomposition of $\widetilde{\bf A}^m$ as ${\bf C}_{ij}^{(m)}$.
Let $\mc D$ denote all indices $i$ such that $V_i$ is a dominant SCC of $G(\widetilde{\bf A}^m)$. We can further decompose this set as follows. 

\begin{definition}
Let $V_i$ be a dominant SCC of $G(\widetilde{\bf A}^m)$. We say that $V_i$ is \emph{pure} if there exists a matrix ${\bf D} \in \widetilde{\C{A}}^m$ with ${\bf D}|_{W_i} = {\bf C}^{(m)}_{ii}$ and we call $V_i$ \emph{mixed} otherwise. The corresponding sets of indices are denoted by $\mc {D_P}$ and $\mc {D_M}$, respectively. 
\end{definition}

Recall that if $V_i$ is mixed, there is a $0 < c_i < 1$ with ${\bf D}|_{V_i} \leqslant c_i \cdot {\bf C}^{(m)}_{ii}$ (element-wise) for all ${\bf D} \in \C{A}^m$, due to Lemma~\ref{LEM:block-dichotomy}. 

The general property $\rho^*(\widetilde{\C{A}}) \leqslant \rho(\widetilde{\bf A})$ relies on the observation that $\widetilde{\bf A}_i \leqslant \widetilde{\bf A}$ element-wise for all $\widetilde{\bf A}_i \in \C{A}$, meaning that $\widetilde{\bf A}$ dominates all elements in $\widetilde{\C{A}}$ in an appropriate sense. If $i \in \mc {D_M}$, we can reduce the block ${\bf C}_{ii}^{(m)}$ by a factor without violating this domination property. 

\begin{definition}
Let $\widetilde{\bf A}'$ be the matrix that is formed from $\widetilde{\bf A}^m$ by replacing ${\bf C}_{ii}^{(m)}$ with $c_i \cdot {\bf C}_{ii}^{(m)}$ for all $i \in \mc {D_M}$. We call $\widetilde{\bf A}'$ the \emph{reduced dominator} of $\widetilde{\C{A}}^m$.
\end{definition}

By construction, $\rho^*(\widetilde{\C{A}}^m) \leqslant \rho(\widetilde{\bf A}')$. On the other hand, the modification from ${\bf C}_{ii}^{(m)}$ to $c_i\cdot {\bf C}_{ii}^{(m)}$ reduces the spectral radius of that block and hence, $\rho(\widetilde{\bf A}^m) = \rho(\widetilde{\bf A}')$ if and only if $\mc {D_P} \neq \varnothing$; this means that at least one of the dominant blocks remains untouched.

\begin{corollary}
Suppose that $\widetilde{\C{A}} = \{ \widetilde{\bf A}_0, \ldots, \widetilde{\bf A}_{k-1}\}$ is a finite set of nonnegative matrices with $\widetilde{\bf A} = \sum_{j=0}^{k-1} \widetilde{\bf A}_j$. Then $\rho^*(\widetilde{\C{A}}) = \rho( \widetilde{\bf A})$ if and only if $\mc {D_P}$ is nonempty. In this case, there exists a ${\bf D} \in \widetilde{\C{A}}^m$ with $\rho^*(\widetilde{\C{A}})= \rho({\bf D})^{1/m}$.
\end{corollary}

\begin{proof}
If $\mc {D_P}$ is nonempty, there exists a ${\bf D} \in \widetilde{\C{A}}^m$ such that ${\bf D}|_{V_i} = {\bf C}_{ii}^{(m)}$, where $\rho({\bf C}_{ii}^{(m)}) = \rho(\widetilde{\bf A}^m)$. Hence,
\[
\rho({\bf D}) \leqslant \rho^*(\widetilde{\C{A}})^m \leqslant \rho(\widetilde{\bf A})^m \leqslant \rho({\bf D}),
\]
which implies equality. On the other hand, if $\mc {D_P} = \varnothing$, we have
\[
\rho^*(\widetilde{\C{A}})^m = \rho^*(\widetilde{\C{A}}^m) \leqslant \rho(\widetilde{\bf A}') < \rho(\widetilde{\bf A})^m,
\]
separating the joint spectral radius of $\widetilde{\C{A}}$ from the spectral radius of $\widetilde{\bf A}$.
\end{proof}

In particular, we have now proven Corollary \ref{cor:fc}, that whenever $\rho^*(\widetilde{\C{A}}) = \rho(\widetilde{\bf A})$ for nonnegative matrices, the finiteness conjecture holds. The assumptions above do not quite suffice to conclude the stronger property that \emph{only} eventually periodic sequences of matrices can maximise the exponential growth rate of the norms.

\begin{example}
Let $\widetilde{\C{A}} = \{ \widetilde{\bf A}_0, \widetilde{\bf A}_1\}$ with
\[
\widetilde{\bf A}_0 = \begin{pmatrix}
1 & 1 & 0\\
0 & 0 & 1\\
0 & 0 & 0
\end{pmatrix}\qquad\mbox{and}\qquad
\widetilde{\bf A}_1 = \begin{pmatrix}
1 & 0 & 0\\
0 & \rho & 0\\
0 & 0 & \rho
\end{pmatrix},
\]
for some $\rho > 2$, which then also coincides with the spectral radius of the sum matrix $\widetilde{\bf A} = \widetilde{\bf A}_0 + \widetilde{\bf A}_1$. The graph $G(\widetilde{\mc A})$ is depicted below (without weights).
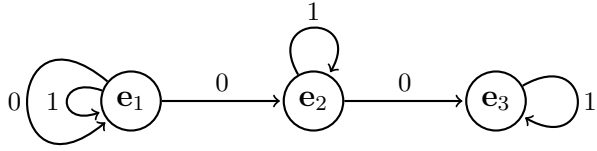
\begin{center}
\begin{tikzpicture}[->, shorten >=1pt, thick, main node/.style={circle,draw,font=\Large\bfseries}, scale = 0.8]

\node[main node](1) at (0,0) {${\bf e}_1$};
\node[main node](2) at (3,0) {${\bf e}_2$};
\node[main node](3) at (6,0) {${\bf e}_3$};

\draw (3) to[out = 30, in = -30, looseness=7] node[right, xshift= - 0.05cm] {$1$} (3) ;

\draw (1) to node[above]{$0$} (2);
\draw (2) to node[above]{$0$} (3);

\draw (1) to[out = 160, in = 200, looseness=6] node[left, xshift= 0.03cm] {$1$} (1);
\draw (1) to[out = 140, in = 220, looseness=9] node[left, xshift= 0.05cm] {$0$} (1);

\draw (2) to[out = 120, in = 60, looseness=6] node[above, yshift= -0.03cm] {$1$} (2);

\end{tikzpicture}
\end{center}
Since $\rho(\widetilde{\bf A}_1) = \rho(\widetilde{\bf A})$, we observe $\rho^*(\widetilde{\C{A}}) = \rho(\widetilde{\bf A}) = \rho$. A direct calculation yields
\[
{\bf B}_n: = \widetilde{\bf A}_1^n \widetilde{\bf A}_0 = 
\begin{pmatrix}
1 & 1 & 0\\
0 & 0 & \rho^n\\
0 & 0 & 0
\end{pmatrix},
\quad\mbox{and}\quad
{\bf B}_n {\bf B}_m  = 
\begin{pmatrix}
1 & 1 & \rho^m\\
0 & 0 & 0\\
0 & 0 & 0
\end{pmatrix}
=: {\bf C}_m,
\] 
from which we see that
\[
{\bf B}_{n_1} \cdots {\bf B}_{n_j} = {\bf C}_{n_j}.
\]
Also, we have
\[
{\bf C}_m \widetilde{\bf A}_1^r = \begin{pmatrix}
1 & \rho^r & \rho^{m+r}\\
0 & 0 & 0\\
0 & 0 & 0\\
\end{pmatrix}.
\]
Let $(n_j)_{j \in \N}$ be a sequence in $\N$ that grows fast enough to satisfy
\[
\lim_{j \to \infty} \frac{1}{n_j} \sum_{i = 1}^{j} n_i = 1.
\]
For notational convenience, we set $N_j := \sum_{i = 1}^{j} n_i$ and $N_0 := 0$.
The product
\[
{\bf B}_{n_1} {\bf B}_{n_2} {\bf B}_{n_3} \cdots
\]
corresponds to a sequence of matrices $(\widetilde{\bf A}_{i_j})_{j \in \N}$ in $\widetilde{\C{A}}$ that is not eventually periodic. Now, note that every $n \in \N$ is of the form $n = N_j + r$ for some $j \in \N_0$ and $0 \leqslant r < n_{j+1}$.
Choosing the matrix norm that corresponds to the maximal column sum (induced by $\|\cdot\|_1$), we obtain
\[
\| \widetilde{\bf A}_{i_1} \cdots \widetilde{\bf A}_{i_n}\| = \| {\bf B}_{n_1} \cdots {\bf B}_{n_j}  \widetilde{\bf A}^r\| = \| {\bf C}_{n_j} \widetilde{\bf A}^r\| = \rho^{n_j + r}
\]
and hence
\[
\| \widetilde{\bf A}_{i_1} \cdots \widetilde{\bf A}_{i_n}\|^{1/n} = \rho^{(n_j + r)/(N_j + r)} \geqslant \rho^{n_j/N_j} \xrightarrow{j \to \infty} \rho.
\]
Hence
\[
\liminf_{n \to \infty} \| \widetilde{\bf A}_{i_1} \cdots \widetilde{\bf A}_{i_n}\|^{1/n} \geqslant \rho.
\]
The same upper bound for the limit superior follows from the fact that $\rho=\rho(\widetilde{\bf A})$, so we have equality.\exend
\end{example}

To understand a bit of the intuition behind the above example, note that there are three SCCs (consisting of a single vertex each), feeding into each other in a linear order. The two terminal SCCs are dominant and require the matrix $\widetilde{\bf A}_1$ in order to ``activate". Both $\widetilde{\bf A}_0$ and $\widetilde{\bf A}_1$ allows one to remain in the first SCC and $\widetilde{\bf A}_0$ allows one to switch between SCCs. We have then constructed a sequence such that for every partial product $P$ the largest suffix $S$ containing only one switch makes up an increasingly large fraction of $P$. Hence, we may spend this suffix entirely in the dominant SCCs (with one switch) and the prefix in the first SCC.
We can slightly adapt this idea to obtain an example where the norm of the product of length $n$ grows even (polynomially) faster than $\rho^n$. For example, we may add one more entry $1$ in $\widetilde{\bf A}_1$ to switch between the two dominant SSCs and get a linear factor. If $n_j$ grows fast enough, this overcompensates the slower growth from the (small) prefix.

From the discussion above, it appears intuitive that the mechanism that allows the liminf to be fast is the ability to switch between two dominant SSCs. Indeed, excluding this possibility rules out that aperiodic products have the maximal exponential rate on all subsequences.
As before, let $m'$ be minimal such that $\widetilde{\bf A}^{m'}$ is dominant positive and $m = 3m'$.

\begin{proposition}
\label{PROP:aperiodic-products}
Suppose that $\rho^*(\widetilde{\C{A}}) = \rho(\widetilde{\bf A})$ and that there is no path in $G(\widetilde{\bf A}^m)$ from one dominant SCC $V_i$, with $i \in \mc {D_P}$, to another. Then, for every aperiodic sequence $(\widetilde{\bf A}_{i_j})_{j \in \N}$, we have
\[
\liminf_{n \to \infty} \|\widetilde{\bf A}_{i_1} \cdots \widetilde{\bf A}_{i_n}\|^{1/n} < \rho^*(\widetilde{\C{A}}).
\]
In particular, this holds if $\rho^*(\widetilde{\C{A}})$ is a simple eigenvalue of $\widetilde{\bf A}$.
\end{proposition}

\begin{proof}
Regrouping the matrices of an aperiodic sequence $(\widetilde{\bf A}_{i_j})_{j \in \N}$ into products of length $m$, we obtain a sequence $({\bf D}_{i_j})_{j \in \N}$ of matrices in $\widetilde{\C{A}}^m$. If this sequence is eventually constant, say equal to ${\bf D}$, then there must be $(i_1,\ldots,i_m) \neq (i_1',\ldots, i_m')$ such that
\[
{\bf D} = \widetilde{\bf A}_{i_1} \cdots \widetilde{\bf A}_{i_m} = \widetilde{\bf A}_{i'_1} \cdots \widetilde{\bf A}_{i'_m}.
\]
This enforces $2{\bf D} \leqslant \widetilde{\bf A}^m$ and the product cannot have the maximal exponential growth. Hence, we may assume that $({\bf D}_{i_j})_{j \in \N}$ is not eventually constant. 
Recall that $\widetilde{\bf A}^m$ has an upper triangular block decomposition which is inherited by each of the matrices in $\widetilde{\C{A}}^m$ and also by all of their products. For ${\bf D}_i \in \widetilde{\C{A}}^m$, we denote the $(s,t)$-block by ${\bf D}_{i,st}$, that is, we decompose
\[
{\bf D}_i = \begin{pmatrix}
{\bf D}_{i,11} & \cdots & {\bf D}_{i,1r}\\
\vdots & \ddots & \vdots\\
0 & \cdots & {\bf D}_{i,rr}.
\end{pmatrix}
\]
Since the structure is the same for all matrices in $\widetilde{\C{A}}^m$, the blocks behave like simple entries under matrix multiplication and the structure is preserved under taking products. For such matrices, we take the notational liberty to refer to blocks in the same way that one usually refers to entries. With this convention,
\[
({\bf D}_{i_1} \cdots {\bf D}_{i_{n}})_{st} = \sum_{q_1, q_2, \ldots, q_{n-1} = 1}^r {\bf D}_{i_1,s q_1} {\bf D}_{i_2, q_1 q_2} \cdots {\bf D}_{i_{n}, q_{n-1} t}.
\]
By the upper triangular structure, all the non-vanishing summands fulfill the additional requirement that $s \leqslant q_1 \leqslant q_2 \leqslant \ldots \leqslant q_n \leqslant t$. In particular, the number of nontrivial summands is bounded by the number of possibilities to choose at most $r-1$ positions of increase from the $n$ matrices, and hence by $n^r$. Using formal concatenation, we may represent any increasing sequence $(q_1, \ldots, q_n)$ by
\[
q_1 \cdots q_n = 1^{p_1} 2^{p_2} \cdots r^{p_r} 
\]
with $0 \leqslant p_q \leqslant n$, where $q^0$ means that the element $q$ is omitted. Since there is no path between dominant pure SCCs, $p_i > 0$ is possible for at most one choice of $i \in \mc {D_P}$, whenever the indices correspond to a non-vanishing summand.

Let $\widetilde{\bf A}''$ be the matrix that emerges from the reduced dominator $\widetilde{\bf A}'$ of $\widetilde{\mc A}^m$ by replacing all dominant diagonal blocks ${\bf C}_{ii}^{(m)}$ with $i \in \mc {D_P}$ by the $0$-block, that is, we deplete all dominant blocks that have not already been reduced in the step $\widetilde{\bf A}^m \to \widetilde{\bf A}'$. Clearly, $\rho(\widetilde{\bf A}'') < \rho(\widetilde{\bf A}^m)$, and for ${\bf D} \in \C{A}^m$ we have ${\bf D}_{st} \leqslant \widetilde{\bf A}''_{st}$, unless $st = ii$ for some $i \in {\mc D}_{\mc P}$.

Since we have assumed that $({\bf D}_{i_n})_{n \in \N}$ is not eventually constant, there are infinitely many positions $n$ such that ${\bf D}_{i_{n}} \neq {\bf D}_{i_n + 1}$. Fix such an $n$, and consider any of the nontrivial summands contributing to $({\bf D}_{i_1} \cdots {\bf D}_{i_{2n}})_{st}$, indexed by some
\[
q_1 \cdots q_{2n} = 1^{p_1} \cdots r^{p_r}.
\]
If none of the $i \in \mc {D_P}$ fulfills $p_i > 0$, we can bound each of the blocks by the corresponding blocks of $\widetilde{\bf A}''$. Choosing the sum of all elements as the norm $\|\cdot\|$, the norm of the product is bounded by $\|( \widetilde{\bf A}'')^{2n}_{st} \| \leqslant \|( \widetilde{\bf A}'')^{2n} \| $. 

The remaining case is that $p_i >0$ for precisely one $i \in \mc {D_P}$. We consider several subcases. First, assume that the sequence $i^{p_i}$ is completely contained in $q_1 \cdots q_n$. Then, to obtain an upper bound, we replace all blocks in the first half of the product by the corresponding block of $\widetilde{\bf A}^m$ and every block in the second half of the product by $\widetilde{\bf A}''$. The norm is hence bounded by
\[
\|\widetilde{\bf A}^{nm}_{sq_n}\| \cdot \| (\widetilde{\bf A}'')^{n}_{q_{n+1}t} \| \leqslant \|\widetilde{\bf A}^{nm}\| \cdot \| (\widetilde{\bf A}'')^{n} \|.
\]
The same estimate holds if $i^{p_i}$ is completely contained in $q_{n} \cdots q_{2n-1}$. It remains to consider the case $q_{n-1} = q_n = q_{n+1} = i$. Since $i_n \neq i_{n+1}$, at least one of the blocks ${\bf D}_{i_n,ii}$ or ${\bf D}_{i_{n+1},ii}$ is equal to the $0$-block (because $i \in \mc {D_P}$) and hence the whole product vanishes.
Regarding all possible cases and keeping in mind that there are less than $n^r$ nontrivial summands, we obtain that
\[
\|({\bf D}_{i_1} \cdots {\bf D}_{i_{2n}})_{st}\| \leqslant n^r \|\widetilde{\bf A}^{nm}\| \cdot \| (\widetilde{\bf A}'')^{n} \|.
\] 
Summing over all $s,t \in \{1,\ldots,r\}$ adds only a factor $r^2$ to the upper bound. Since this estimate holds for infinitely many $n$, we find
\begin{align*}
\liminf_{n \to \infty} \|{\bf D}_{i_1} \cdots {\bf D}_{i_{2n}}\|^{1/2n} & \leqslant \liminf_{n \to \infty} \|\widetilde{\bf A}^{nm}\|^{1/2n}\cdot \|(\widetilde{\bf A}'')^n \|^{1/2n}\\ 
&\leqslant (\rho(\widetilde{\bf A}'') \rho(\widetilde{\bf A}^m))^{1/2}\\
&< \rho(\widetilde{\bf A}^m).
\end{align*}
Recalling that $({\bf D}_{i_n})_{n \in \N}$ emerges from a regrouping of the original sequence into blocks of length $m$, we get
\[
\liminf_{n \to \infty} \| \widetilde{\bf A}_{i_1} \cdots \widetilde{\bf A}_{i_n}\|^{1/n}
\leqslant \liminf_{n \to \infty} \|{\bf D}_{i_1} \cdots {\bf D}_{i_n}\|^{1/mn} < \rho(\widetilde{\bf A}),
\]
which shows the first claim.
If $\rho^{\ast}(\widetilde{\bf A})$ is a simple eigenvalue of $\widetilde{\bf A}$, then $G(\widetilde{\bf A})$ has only one dominant SCC. Taking the power $m$ may split up this SCC into several smaller components which are, however, not connected in $G(\widetilde{\bf A}^m)$. Hence, there are no paths between dominant SCCs in this case.
\end{proof}

In the following, we show that assuming that there is no path between different pure dominant SCCs is no restriction as long as we consider the action on an $\widetilde{\bf A}$-eigenvector with an eigenvalue of maximal modulus. To this end, we introduce one more bit of graph-theoretic notation.

\begin{definition}
The \emph{backward closure} of an SCC $V_i$, denoted by $\mc B(i)$, is the set of all vertices $j$ such that there exists a path from $j$ to $i$ in $G(\widetilde{\bf A})$. 
\end{definition}

Clearly, a backward closure consists of complete sets of SCCs. If $V_i$ is nontrivial, we observe that $V_i \subset \mc B(V_i)$. Given several SCCs, $V_{i_1},\ldots,V_{i_n}$ we call $\mc B(i_1,\ldots, i_n) := \bigcup_{j=1}^n \mc B(i_j)$ the backward closure of the corresponding collection. The corresponding subspace is given by
\[
B(i_1,\ldots, i_n) : = {\rm span} \{{\bf e}_j \}_{j \in \mc B(i_1,\ldots,i_n)}.
\] 

\begin{lemma}
\label{LEM:backward-closure-invariance}
The backward closure subspace $B(i)$ of an SCC $V_i$ is invariant under the action of $\widetilde{\bf A}$, that is, $\widetilde{\bf A} B(i) \subset B(i)$. Moreover, the same holds for the backward closure of an arbitrary collection.
\end{lemma}

\begin{proof}
Since ${\bf e}_s^T \widetilde{\bf A} {\bf e}_t > 0$ if and only if there is an edge from $s$ to $t$ in $G(\widetilde{\bf A})$, we observe that, in this case, a path from $t$ to $V_i$ also implies the existence of a path from $s$ to $V_i$. Hence, if $t \in \mc B(i)$, the vector $\widetilde{\bf A} {\bf e}_t$ is contained in $B(i)$. The same holds of course for linear combinations of ${\bf e}_j$, with $j \in \mc B(i)$. The last claim follows from the fact that $B(i_1,\ldots,i_n)$ is a direct sum of the subspaces $B(i_1),\ldots,B(i_n)$.
\end{proof}

\begin{definition}
We call a dominant SCC $V_i$ \emph{initial}, if there is no path from another dominant $V_j$ to $V_i$. We write $\ID$ for the union of initial dominant SCCs.
\end{definition}

\begin{lemma}
\label{LEM:backward closure-support}
Every nonnegative $\rho$-eigenvector ${\bf v}$ of $\widetilde{\bf A}$ is supported on the backward closure subspace $B(\ID)$.
\end{lemma}

\begin{proof}
First, note that taking a power of $\widetilde{\bf A}$ leaves the set of all initial dominant vertices invariant (although it might decompose the corresponding SCCs). The same holds for the corresponding backward closure. Hence, we may replace $\widetilde{\bf A}$ by a higher power whenever convenient.

Suppose that $\widetilde{\bf A}{\bf v} = \rho{\bf v}$ and write ${\bf v} = \sum_{i = 1}^r {\bf v}_i$, with ${\bf v}_i$ supported on the SCC $V_i$. 
Writing ${\bf C}_{ij}$ for the $(i,j)$-block of $\widetilde{\bf A}$, we get
\[
\rho\, {\bf v}_i = \sum_{j\geqslant i} {\bf C}_{ij} {\bf v}_j.
\]
We say that $i$ is ${\bf v}$-final if ${\bf v}_i \neq 0$ and ${\bf C}_{ij} {\bf v}_j = 0$ for all $j > i$, and we denote the corresponding set of indices by $\mc F({\bf v})$. 
Observe that ${\bf v}$ is supported on $B(\mc F({\bf v}))$. Indeed, each ${\bf v}_i \neq 0$ has either $i \in \mc F({\bf v})$ or there is an edge from $i$ to some $j$ with $j>i$. The same reasoning applies to the index $j$, and we obtain a path to a ${\bf v}$-final index after at most $r-1$ steps. 

It remains to show that $\mc F({\bf v}) \subset \ID$ since this then implies that ${\bf v}$ is supported on $B(\mc F({\bf v})) \subset B(\ID)$.
For $i \in \mc F({\bf v})$, we have 
\[
\rho\, {\bf v}_i = {\bf B}_{ii} {\bf v}_i,
\]
and hence $V_i$ must be dominant. For a moment suppose that $V_i$ is not initial. Then, there exists another dominant SCC $V_s$ and a path $s \to i$. Possibly taking a higher power of $\widetilde{\bf A}$, we may assume that ${\bf C}_{si},{\bf C}_{ii},{\bf C}_{ss}$ are strictly positive.
Using the block-structure of $\widetilde{\bf A}$, the nonnegativity of ${\bf v}$ and ${\bf C}_{ii} {\bf v}_i = \rho \, {\bf v}_i$, we get
\[
{\bf v}_s = 
\frac{1}{\rho^k}(\widetilde{\bf A}^k {\bf v})_s
\geqslant \frac{1}{\rho^k} \sum_{j = 0}^{k-1} {\bf C}_{ss}^j {\bf C}_{si} {\bf C}_{ii}^{k-j-1} {\bf v}_i
= \sum_{j = 0}^{k-1} \frac{{\bf C}_{ss}^j}{\rho^j} \cdot\frac{{\bf C}_{si}}{\rho} {\bf v}_i,
\]
which is unbounded, since, by standard Perron--Frobenius theory, $\rho^{-j}{\bf C}_{ss}^j$ converges to a projector with strictly positive entries. This contradiction shows that $V_i$ indeed needs to be initial. The claim follows.
\end{proof}

\begin{remark}
The assumption that ${\bf v}$ is nonnegative is essential for the validity of Lemma~\ref{LEM:backward closure-support}. This can be seen from the example
\[
\widetilde{\bf A} = \begin{pmatrix}
1 & 1 & 1\\
0 & 1 & 0\\
0 & 0 & 1
\end{pmatrix},
\]
which has an eigenvector ${\bf v} = (0,1,-1)^T$ that is {\em not} supported on the backward closure of the initial dominant ${\bf e}_1$. The reason behind this is that different signs make it possible to cancel out the influx from other dominant components to a given initial dominant component. \exend
\end{remark}

\begin{lemma}
\label{LEM:aperiodic-vanishing}
Suppose that $\rho:=\rho^*(\widetilde{\C{A}}) = \rho(\widetilde{\bf A})$, and that $ {\bf v} \in \R^{d'}$ is a nonnegative $\rho$-eigenvector of $\widetilde{\bf A}$. Then, for every aperiodic sequence $(\widetilde{\bf A}_{i_j})_{j \in \N}$, we have
\[
\liminf_{n \to \infty} \frac{1}{\rho^{n}} \| \widetilde{\bf A}_{i_1} \cdots \widetilde{\bf A}_{i_{n}} {\bf v} \| = 0.
\]
\end{lemma}

\begin{proof}
For a moment, assume that the limit inferior is positive. Then, we get 
\begin{equation}
\label{EQ:liminf-max}
\liminf_{n \to \infty} \| \widetilde{\bf A}_{i_1} \cdots \widetilde{\bf A}_{i_n} {\bf v}\|^{1/n} = \rho.
\end{equation}
Since ${\bf v}$ is supported on $B(\mc {D_I})$ (Lemma~\ref{LEM:backward closure-support}) and this set is invariant under each $\widetilde{\bf A}_j$ (Lemma~\ref{LEM:backward-closure-invariance}), it suffices to consider the restriction of the family $\widetilde{\mc A}$ to $B(\mc {D_I})$. If this family has a joint spectral radius strictly smaller than $\rho$ we directly obtain a contradiction to \eqref{EQ:liminf-max}. Otherwise, since there is no path between dominant SCCs in $\mc B(\mc {D_I})$ (and since this property is stable under taking powers of $\widetilde{\bf A}$), the conditions of Proposition~\ref{PROP:aperiodic-products} are satisfied for the restricted family and we again reach a contradiction.
\end{proof}

The following result directly implies Theorem \ref{thm:mainrational} stated in the Introduction.

\begin{proposition}
\label{PROP:no-irrational-masses}
Let $\mu$ be a ghost measure of a positively presentable, $k$-regular sequence $f$ and $y \notin \Q \cap [0,1]$. Then, $\mu(\{y\}) = 0$.
\end{proposition}

\begin{proof}
If $\rho^{\ast}(\widetilde{\mc A}) < \rho(\widetilde{\bf A})$,  then the same estimate holds for $\rho^{\ast}(\widehat{\mc A}) \leqslant \rho^{\ast}(\widetilde{\mc A})$ and the measure $\mu$ is continuous by Theorem~\ref{PROP:continuity-condition}. We may therefore restrict to the case $\rho = \rho^{\ast}(\widetilde{\mc A}) = \rho(\widetilde{\bf A})$ in the following.
Let $q$ be the period of $\widetilde{\bf A}$ and $\mu = \mu_h$ with $h \in G$. Recall that in the case of nonnegative matrices, $G$ is a $q$-cyclic group, generated by the element $g$. Hence, $h$ is of the form $h = g^j$ for some $0\leqslant j \leqslant q-1$. Due to the assumption that $y$ is irrational, it has a unique, aperiodic coding sequence $x \in \pi^{-1}(y)$.
By Corollary~\ref{COR:point-mass-sequence} (and the observation that ${\bf u}^T {\bf A}_{x_1} \in V$), there exists a constant $c>0$ such that 
\[
\mu(\{y\}) \leqslant \liminf_{n \to \infty} \frac{c}{\rho^n} \|\widetilde{ \bf A}_{x_2} \ldots  \widetilde{\bf A}_{x_{n+1}} {\bf R}^n {\bf R}_h {\bf P} {\bf v}\|. 
\]
Since ${\bf R}^q$ is the identity matrix, it follows by 
\[
{\bf P}= \lim_{n \to \infty} c_{qn} \widetilde{\bf A}^{qn},
\]
and the nonnegativity of $\bf v$ that ${\bf v}' = {\bf P} {\bf v}$ is also a nonnegative vector. By the same expression for ${\bf P}$, we also observe that ${\bf v}'$ is a $\rho^q$-eigenvector for $\widetilde{\bf A}^q$. For the rotation matrices we get 
\[
{\bf R}_h = {\bf R}_{g^j} = {\bf R}^{-j},
\]
due to the fact that ${\bf R} = {\bf R}_{g^{-1}}$. In particular, for $n_m = qm + j$, we obtain
\[
{\bf R}^{n_m} {\bf R}_h = {\bf R}^{qm} = {\bf 1},
\]
the identity matrix. Hence, for $j'=j+q$,
\begin{align*}
\mu(\{y\})& \leqslant \liminf_{m \to \infty} \frac{c}{\rho^{qm+j'}} \|\widetilde{\bf A}_{x_2} \cdots \widetilde{\bf A}_{x_{qm+j'}} {\bf v}'\|
\\ &\leqslant \frac{c}{\rho^{j'-1}} \|\widetilde{\bf A}_{x_2} \cdots \widetilde{\bf A}_{x_{j'}}\| \liminf_{m \to \infty} \frac{1}{\rho^{qm}} \|\widetilde{\bf A}_{x_{1+j'}} \cdots \widetilde{\bf A}_{x_{qm + j'}} {\bf v}'\|.
\end{align*}
At this point, Lemma~\ref{LEM:aperiodic-vanishing} is applicable to the collection $\widetilde{\mc A}^q$ with sum-matrix $\widetilde{\bf A}^q$, spectral radius $\rho^q$ and eigenvector ${\bf v}'$ because the corresponding regrouping of indices remains aperiodic. This proves $\mu(\{y\}) = 0$, as desired. 
\end{proof}

\begin{example}\label{EX:pp+ac}
Since the elements $C_n(h)$ which determine the ghost measure $\mu_h$ have such a similar structure, it is natural to ask whether all ghost measures of a fixed sequence $f$ are in fact equivalent, maybe under the additional requirement that $f$ is positively presentable. As we will see however, this is not even the case for the pure point part of the ghost measures. 

For example, consider a collection $\{{\bf A}_0, {\bf A}_1\}$ with the following graph,
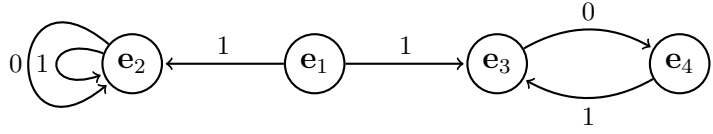
\begin{center}
\begin{tikzpicture}[->, shorten >=1pt, thick, main node/.style={circle,draw,font=\Large\bfseries}, scale = 0.8]

\node[main node](1) at (0,0) {${\bf e}_2$};
\node[main node](2) at (3,0) {${\bf e}_1$};
\node[main node](3) at (6,0) {${\bf e}_3$};
\node[main node](4) at (9,0) {${\bf e}_4$};

\draw (2) to node[above]{$1$} (1);
\draw (2) to node[above]{$1$} (3);

\draw (1) to[out = 160, in = 200, looseness=8] node[left, xshift= 0.03cm] {$1$} (1);
\draw (1) to[out = 140, in = 220, looseness=9] node[left, xshift= 0.05cm] {$0$} (1);

\draw (3) to[out = 30, in = 150] node[above] {$0$} (4);
\draw (4) to[out = 210, in = 330] node[below] {$1$}(3);

\end{tikzpicture}
\end{center}
We choose unit weights except on the self-loops on ${\bf e}_2$, which each get weight $1/2$. This ensures that the dominant SCCs are given by $\{{\bf e}_2\}$ and $\{ {\bf e}_3, {\bf e}_4 \}$. 
Intuitively, the left part of the graph corresponds to an absolutely continuous component, whereas the right part is responsible for pure points.
We may choose an ${\bf A}^2$-eigenvector ${\bf v}$ (to the dominant eigenvalue $\rho = 1$) such that the last component vanishes, for example ${\bf v} = (1,1,1,0)^T$. Let $x = (10)^{\infty}$ and observe that ${\bf e}_3$ can be reached from ${\bf u} = {\bf e}_1$ only along paths with an odd length. Take $f$ to be the sequence with representation $({\bf u}, \{{\bf A}_0, {\bf A}_1\}, {\bf v})$. 
 In this example we can calculate both ghost measures explicitly. Let $n \in \N$ and $k^n \leqslant m < k^{n+1}$. Via a direct calculation (or tracing the available paths in the above diagram) we obtain that
\[
{\bf u}^T {\bf A}_{(m)_k} {\bf v} = 
\begin{cases}
1 + 2^{-n}, & \mbox{if } n \in 2\N \mbox{ and } (m)_k = (10)^{n/2} 1,
\\ 2^{-n}, & \mbox{otherwise}.
\end{cases}
\]
Since this coincides with $\mu_n(I_{n,m})$ (up to the rescaling to a probability measure), we obtain that $\mu_n$ approaches Lebesgue measure if $n$ is odd, that is $\mu_{f,1} = \operatorname{Leb}$. Similarly, we get
\[
\mu_{f,2} = \frac{1}{2}(\operatorname{Leb} + \delta_{\pi(x)}),
\]
which is a ghost measure of mixed spectral type. \exend
\end{example}

\section{Concluding remarks}\label{sec:conc}

The immediate questions that arise concern the necessity of various assumptions. 

The nondegeneracy assumption for a regular sequence is strictly stronger than required for the approximants $\mu_N$ to be well-defined. If the sequence is degenerate, smaller eigenvalues of the matrix $\widetilde{\bf A}$ can be expected to enter the construction of an appropriate group parametrisation. Whether or not they do may depend on the Diophantine properties of the angles that occur in the peripheral spectrum. There is a lot of room for exploration here.

It may be that not all relevant lead Jordan eigenvectors appear in ${\bf P}{\bf v}$. If this is not the case, there is a nontrivial subgroup of $G$ that gives a group parametrisation. If all measures are continuous, it should be straightforward to show that an appropriate restriction makes the parametrisation one-to-one, turning it into a homeomorphism. This endows the set of $g$-measures with a group structure. Does this hold without the continuity assumption, or maybe even under weaker nondegeneracy assumptions?

Under the strictest of conditions, one has the nicest looking results, especially considering spectral properties. In particular, our results prove the following characterisation.

\begin{proposition} Let $f$ be a regular sequence and suppose that the set $\mathcal{A}$ is irreducible and that ${\bf A}$ is nonnegative and primitive. Then there is a unique ghost measure $\mu_f$, which is spectrally pure. Moreover, 
\begin{enumerate}
\item[(i)] $\rho^*=\rho$ if, and only if, $\mu_f$ is pure point,
\item[(ii)] $\rho/k<\rho^*\neq\rho$ if, and only if, $\mu_f$ is singular continuous and 
\item[(iii)] $\rho/k=\rho^*$ if, and only if, $\mu_f$ is absolutely continuous.
\end{enumerate}
\end{proposition}

\noindent But, what are the minimal conditions for spectral characteristics to be the same for {\em all} ghost measures? In particular, is there a natural (weak) assumption for which the above proposition holds for the set of all ghost measures, that is, when there is not a unique ghost measure?

Concerning the pure points of ghost measures, we ask, can maximal growth at irrational positions appear, if we drop the nonnegativity assumption of the matrices? Disproving this might necessitate a deeper dive into more general versions of our adapted ``finiteness property'' result.

As a final thought, one could ask about the multifractal analysis. Here, under appropriate assumptions on the matrices, one should be able to harvest the results of Feng \cite{F2003,F2009}, that give a `fine resolution' for the ghost measures. Is it uniform over all ghost measures, under what conditions, and, what does it tell us about the original sequence? At the moment, this remains a mystery.

\section*{Acknowledgements}

It is our pleasure to thank Michael Baake for helpful conversations and comments, and for suggesting important references. MC acknowledges the support of the German Research Foundation (DFG) via SFB 1283/2 2021–317210226, JE acknowledges the support of the Commonwealth of Australia, PG
acknowledges support from the German Research Foundation (DFG) through grant GO
3794/1-1 and NM is supported by the German Academic Exchange Service (DAAD) through a Postdoctoral Researchers International Mobility Experience (PRIME) Fellowship.

\bibliographystyle{amsplain}
\providecommand{\bysame}{\leavevmode\hbox to3em{\hrulefill}\thinspace}
\providecommand{\MR}{\relax\ifhmode\unskip\space\fi MR }
\providecommand{\MRhref}[2]{%
  \href{http://www.ams.org/mathscinet-getitem?mr=#1}{#2}
}
\providecommand{\href}[2]{#2}


\end{document}